\documentclass[oneside,english]{amsart}
\usepackage{lmodern}

\usepackage[T1]{fontenc}
\usepackage[latin9]{inputenc}
\usepackage{enumitem}
\usepackage{amsthm}
\usepackage{amssymb}
\PassOptionsToPackage{normalem}{ulem}
\usepackage{ulem}

\makeatletter
\numberwithin{equation}{section}
\numberwithin{figure}{section}
\theoremstyle{plain}
\newtheorem{thm}{\protect\theoremname}[section]
  \theoremstyle{plain}
  \newtheorem{fact}[thm]{\protect\factname}
 \theoremstyle{plain}
 \newtheorem*{main_thm*}{Main Theorem}
  \theoremstyle{definition}
  \newtheorem{defn}[thm]{\protect\definitionname}
  \theoremstyle{remark}
  \newtheorem{rem}[thm]{\protect\remarkname}
  \theoremstyle{plain}
  \newtheorem{lem}[thm]{\protect\lemmaname}
  \theoremstyle{remark}
  \newtheorem{note}[thm]{\protect\notename}
  \theoremstyle{plain}
  \newtheorem{prop}[thm]{\protect\propositionname}
  \theoremstyle{plain}
  \newtheorem{cor}[thm]{\protect\corollaryname}
  \theoremstyle{remark}
  \newtheorem{claim}[thm]{\protect\claimname}
 \newlist{casenv}{enumerate}{4}
 \setlist[casenv]{leftmargin=*,align=left,widest={iiii}}
 \setlist[casenv,1]{label={{\itshape\ \casename} \arabic*.},ref=\arabic*}
 \setlist[casenv,2]{label={{\itshape\ \casename} \roman*.},ref=\roman*}
 \setlist[casenv,3]{label={{\itshape\ \casename\ \alph*.}},ref=\alph*}
 \setlist[casenv,4]{label={{\itshape\ \casename} \arabic*.},ref=\arabic*}
  \theoremstyle{remark}
  \newtheorem{notation}[thm]{\protect\notationname}
  \theoremstyle{definition}
  \newtheorem{problem}[thm]{\protect\problemname}

\usepackage{multicol}
\usepackage{a4wide}
\usepackage{amssymb}
\usepackage{url}
\makeatother

\usepackage{babel}
  \providecommand{\claimname}{Claim}
  \providecommand{\corollaryname}{Corollary}
  \providecommand{\definitionname}{Definition}
  \providecommand{\factname}{Fact}
  \providecommand{\lemmaname}{Lemma}
  \providecommand{\notationname}{Notation}
  \providecommand{\notename}{Note}
  \providecommand{\problemname}{Problem}
  \providecommand{\propositionname}{Proposition}
  \providecommand{\remarkname}{Remark}
 \providecommand{\casename}{Case}
\providecommand{\theoremname}{Theorem}

\begin{document}
\global\long\def\arity{\operatorname{Ar}}
\global\long\def\C{\mathfrak{C}}
\global\long\def\acl{\operatorname{acl}}
\global\long\def\tp{\operatorname{tp}}
\global\long\def\qf{\operatorname{qf}}
\global\long\def\id{\operatorname{id}}
\global\long\def\Add{\operatorname{Add}}
\global\long\def\ded{\operatorname{ded}}
\global\long\def\cof{\operatorname{cof}}
\global\long\def\SS{\mathcal{P}}
\global\long\def\EM{\operatorname{EM}}
\global\long\def\tr{\operatorname{tr}}
\global\long\def\F{\mathcal{F}}
\global\long\def\tr{\operatorname{tr}}
\global\long\def\lev{\operatorname{lev}}

\def\Ind#1#2{#1\setbox0=\hbox{$#1x$}\kern\wd0\hbox to 0pt{\hss$#1\mid$\hss} \lower.9\ht0\hbox to 0pt{\hss$#1\smile$\hss}\kern\wd0} 
\def\Notind#1#2{#1\setbox0=\hbox{$#1x$}\kern\wd0\hbox to 0pt{\mathchardef \nn="3236\hss$#1\nn$\kern1.4\wd0\hss}\hbox to 0pt{\hss$#1\mid$\hss}\lower.9\ht0 \hbox to 0pt{\hss$#1\smile$\hss}\kern\wd0} 
\def\nind{\mathop{\mathpalette\Notind{}}}

\global\long\def\ind{\mathop{\mathpalette\Ind{}}}
 \global\long\def\nind{\mathop{\mathpalette\Notind{}}}
\global\long\def\indi{\mathop{\mathpalette\Ind{}}}

\global\long\def\nf{\mbox{nf}}
\global\long\def\Uu{\mathcal{U}}
\global\long\def\dom{\operatorname{dom}}
\global\long\def\concat{\frown}

\global\long\def\NTPT{\operatorname{NTP}_{\operatorname{2}}}
\global\long\def\ist{\operatorname{ist}}
\global\long\def\C{\mathbb{M}}
\global\long\def\alt{\operatorname{alt}}
\global\long\def\assum{\smiley}
\global\long\def\leftexp#1#2{{\vphantom{#2}}^{#1}{#2}}

\title{On non-forking spectra }

\author{Artem Chernikov, Itay Kaplan and Saharon Shelah}

\address{A. Chernikov: \'{E}quipe de Logique Math\'{e}matique, Institut de Math\'{e}matiques de Jussieu
- Paris Rive Gauche, B\^{a}timent Sophie Germain, Universit\'{e} Paris Diderot
Paris 7, UFR de Math\'{e}matiques - case 7012, 75205 Paris Cedex 13 France}
 \email{art.chernikov@gmail.com}

\address{I. Kaplan: Institute of Mathematics, The Hebrew University of Jerusalem, Givat
Ram, Jerusalem 91904, Israel}
 \email{kaplan@math.huji.ac.il}

\address{S. Shelah: Institute of Mathematics, The Hebrew University of Jerusalem, Givat
Ram, Jerusalem 91904, Israel and
 Department of Mathematics Hill Center-Busch Campus Rutgers, The State University of New Jersey 110 Frelinghuysen Road Piscataway, NJ 08854-8019 USA}
  \email{shelah@math.huji.ac.il}


\begin{abstract}
Non-forking is one of the most important notions in modern model theory
capturing the idea of a generic extension of a type (which is a far-reaching
generalization of the concept of a generic point of a variety).

To a countable first-order theory we associate its \emph{non-forking
spectrum} --- a function of two cardinals $\kappa$ and $\lambda$
giving the supremum of the possible number of types over a model of
size $\lambda$ that do not fork over a sub-model of size $\kappa$.
This is a natural generalization of the stability function of a theory.

We make progress towards classifying the non-forking spectra. On the
one hand, we show that the possible values a non-forking spectrum
may take are quite limited. On the other hand, we develop a general
technique for constructing theories with a prescribed non-forking
spectrum, thus giving a number of examples. In particular, we answer
negatively a question of Adler whether NIP is equivalent to bounded
non-forking.

In addition, we answer a question of Keisler regarding the number
of cuts a linear order may have. Namely, we show that it is possible
that $\ded\kappa<\left(\ded\kappa\right)^{\omega}$.
\end{abstract}

\maketitle

\section{Introduction}

The notion of a non-forking extension of a type (see Definition \ref{def:forking})
was introduced by Shelah for the purposes of his classification program
to capture the idea of a ``generic'' extension of a type to a larger
set of parameters which essentially doesn't add new constraints to
the set of its solutions. In the context of stable theories non-forking
gives rise to an independence relation enjoying a lot of natural properties
(which in the special case of vector spaces amounts to linear independence
and in the case of algebraically closed fields to algebraic independence)
and is used extensively in the analysis of models. In a subsequent
work of Shelah \cite{SheSimple}, Kim and Pillay \cite{KimForking,KimPillaySimple}
the basic properties of forking were generalized to a larger class
of simple theories. Recent work of the first and second authors shows
that many properties of forking still hold in a larger class of theories
without the tree property of the second kind \cite{cheka}.

Here we consider the following basic question: how many non-forking
extensions can there be? More precisely, given a complete first-order
theory $T$, we associate to it its non-forking spectrum, a function
$f_{T}(\kappa,\lambda)$ from cardinals $\kappa\leq\lambda$ to cardinals
defined as:
\[
f_{T}\left(\kappa,\lambda\right)=\mbox{sup}\left\{ S^{\nf}(N,M)\left|\,M\preceq N\models T,\,|M|\leq\kappa,\,|N|\leq\lambda\right.\right\} ,
\]
where $S^{\mbox{\ensuremath{\nf}}}\left(A,B\right)=\left\{ p\in S_{1}(A)\left|\,p\mbox{ does not fork over }B\right.\right\} $
(counting $1$-types rather than $n$-types is essential, as the value
may depend on the arity, see Section \ref{sub:f_T1 different than f_T2}).

This is a generalization of the classical question ``how many types
can a theory have over a model?''. Recall that the stability function of a theory
is defined as $$f_{T}\left(\kappa\right)= \sup\left\{ S\left(M\right)\left|\,M\models T,\,\left|M\right|=\kappa\right.\right\}. $$
It is easy to see that $f_{T}\left(\kappa,\kappa\right)=f_{T}\left(\kappa\right)$.
This function has been studied extensively by Keisler \cite{KeislerSixClasses}
and the third author \cite{shelahStability}, where the following
fundamental result was proved:
\begin{fact}
\label{fac:Keisler}For any complete countable first-order theory
$T$, $f_{T}$ is one of the following: $\kappa$, $\kappa+2^{\aleph_{0}}$,
$\kappa^{\aleph_{0}}$, $\ded\left(\kappa\right)$, $\ded\left(\kappa\right)^{\aleph_{0}}$,
$2^{\kappa}$.
\end{fact}
Where $\ded\left(\kappa\right)$ is the supremum of the number of
cuts that a linear order of size $\kappa$ may have (see Definition
\ref{def:ded}). While this result is unconditional, in some models
of $ZFC$, some of these functions may coincide. Namely, if $GCH$
holds, $\ded\left(\kappa\right)=\ded\left(\kappa\right)^{\aleph_{0}}=2^{\kappa}$.
By a result of Mitchell \cite{Mitchell}, it was known that for any
cardinal $\kappa$ with $\cof\kappa>\aleph_{0}$ consistently $\ded\left(\kappa\right)<2^{\kappa}$.
In 1976, Keisler \cite[Problem 2]{KeislerSixClasses} asked whether
$\ded\left(\kappa\right)<\ded\left(\kappa\right)^{\aleph_{0}}$ is
consistent with $ZFC$. We give a positive answer in Section \ref{sec:on ded kappa < ded kappa ^aleph0}.

The aim of this paper is to classify the possibilities of $f_{T}\left(\kappa,\lambda\right)$.
The philosophy of ``dividing lines'' of the third author suggests
that the possible non-forking spectra are quite far from being arbitrary,
and that there should be finitely many possible functions, distinguished
by the lack (or presence) of certain combinatorial configurations.
We work towards justifying this philosophy and arrive at the following
picture.
\begin{main_thm*}
\label{MainTheorem}Let $T$ be a countable complete first-order theory.
Then for $\lambda\gg\kappa$, $f_{T}(\kappa,\lambda)$ can be one
of the following, in increasing order (meaning that we have an example
for each item in the list except for (\ref{enu:lambda^beth}), and
``???'' means that we don't know if there is anything between the
previous and the next item, while the lack of ``???'' means that
there is nothing in between):

\begin{multicols}{4}
\begin{enumerate}
\item \label{enu:kappa}$\kappa$
\item \label{enu:kappa+ 2^omega}$\kappa+2^{\aleph_{0}}$
\item \label{enu:kappa^omega}$\kappa^{\aleph_{0}}$
\item \label{enu:ded kappa}$\ded\kappa$
\item \label{enu:??? between k^omega and ded k ^omega}???
\item \label{enu:ded kappa ^ omega}$\left(\ded\kappa\right)^{\aleph_{0}}$
\item \label{enu:double exponent kappa}$2^{2^{\kappa}}$
\item \label{enu:lambda}$\lambda$
\item \label{enu:lambda ^omega}$\lambda^{\aleph_{0}}$
\item \label{enu:??? between lambda^omega and lambda^<beth}???
\item \label{enu:lambda^beth}$\lambda^{<\beth_{\aleph_{1}}\left(\kappa\right)}$
\item \label{enu:ded lambda}$\ded\lambda$
\item \label{enu:??? between ded lambda and ded lambda^omega}???
\item \label{enu:ded lambda ^ omega}$\left(\ded\lambda\right)^{\aleph_{0}}$
\item \label{enu:between ded lambda ^omega and 2^lambda}???
\item \label{enu:2 ^ lambda}$2^{\lambda}$
\end{enumerate}

\end{multicols}

\end{main_thm*}
In particular, note that the existence of an example of $f_{T}\left(\kappa,\lambda\right)=2^{2^{\kappa}}$
answers negatively a question of Adler \cite[Section 6]{Ad} whether
NIP is equivalent to bounded non-forking.

The restriction $\lambda\gg\kappa$ is in order to make the statement
clearer. It can be taken to be $\lambda\geq\beth_{\aleph_{1}}\left(\kappa\right)$.
In fact we can say more about smaller $\lambda$ in some cases. In
the class of $\NTPT$ theories (see Section \ref{sec:Inside NTP2}),
we have a much nicer picture, meaning that there is a gap between
(\ref{enu:ded kappa ^ omega}) and (\ref{enu:2 ^ lambda}). \\

In the first part of the paper, we prove that the non-forking spectra
cannot take values which are not listed in the Main Theorem. The proofs
here combine techniques from generalized stability theory (including
results on stable and NIP theories, splitting and tree combinatorics)
with a two cardinal theorem for $L_{\omega_{1},\omega}$.

The second part of the paper is devoted to examples.

We introduce a general construction which we call \emph{circularization}.
Roughly speaking, the idea is the following: modulo some technical
assumptions, we start with an arbitrary theory $T_{0}$ in a finite
relational language and an (essentially) arbitrary prescribed set
of formulas $F$. We expand $T$ by putting a circular order on the
set of solutions of each formula in $F$, iterate the construction
and take the limit. The point is that in the limit all the formulas
in $F$ are forced to fork, and we have gained some control on the
set of non-forking types. This construction turns out to be quite
flexible: by choosing the appropriate initial data, we can find a
wide range of examples of non-forking spectra previously unknown.

\section{\label{sec:Preliminaries}Preliminaries}

Our notation is standard: $\kappa,\lambda,\mu$ are cardinals; $\alpha,\beta,\ldots$
are ordinals; $M,N,\ldots$ are models; $\C$ is always a monster
model of the theory in question; $B^{\left[\kappa\right]}$ is the
set of subsets of $B$ of size $\leq\kappa$; $T$ is a complete countable
first-order theory; for a sequence $\bar{a}=\left\langle a_{i}\left|\,i<\alpha\right.\right\rangle $,
$\EM\left(\bar{a}/A\right)$ denotes its Ehrenfeucht-Mostowski type
over $A$.

\subsection{Basic properties of forking and dividing.\protect \\
}

We recall the definition of forking and dividing (e.g. see \cite[Section 2]{cheka}
for more details).
\begin{defn}
(Dividing) Let $A$ be be a set, and $a$ a tuple. We say that the
formula $\varphi\left(x,a\right)$ \emph{divide}s over $A$ if and
only if there is a number $k<\omega$ and tuples $\left\{ a_{i}\left|i<\omega\right.\right\} $
such that
\begin{enumerate}
\item $\tp\left(a_{i}/A\right)=\tp\left(a/A\right)$.
\item The set $\left\{ \varphi\left(x,a_{i}\right)\left|\,i<\omega\right.\right\} $
is $k$-inconsistent (i.e. every subset of size $k$ is not consistent).
\end{enumerate}
In this case, we say that a formula $k$-divides.\end{defn}
\begin{rem}
From Ramsey and compactness it follows that $\varphi\left(x,a\right)$
divides over $A$ if and only if there is an indiscernible sequence
over $A$, $\left\langle a_{i}\left|i<\omega\right.\right\rangle $
such that $a_{0}=a$ and $\left\{ \varphi\left(x,a_{i}\right)\left|\,i<\omega\right.\right\} $
is inconsistent.\end{rem}
\begin{defn}
\label{def:forking}(Forking) Let $A$ be be a set, and $a$ a tuple.
\begin{enumerate}
\item Say that the formula $\varphi\left(x,a\right)$ \emph{forks} over
$A$ if there are formulas $\psi_{i}\left(x,a_{i}\right)$ for $i<n$
such that $\varphi\left(x,a\right)\vdash\bigvee_{i<n}\psi_{i}\left(x,a_{i}\right)$
and $\psi_{i}\left(x,a_{i}\right)$ divides over $A$ for every $i<n$.
\item Say that a type $p$ forks over $A$ if there is a finite conjunction
of formulas from $p$ which forks over $A$.
\end{enumerate}
\end{defn}
It follows immediately from the definition that if a partial type
$p\left(x\right)$ does not fork over $A$ then there is a global
type $p'\left(x\right)\in S\left(\C\right)$ extending $p\left(x\right)$
that does not fork over $A$.
\begin{lem}
\label{lem: Cofinal set} Let $(A,\leq)$ be a $\kappa^{+}$-directed
order and let $f:\,A\to\kappa$. Then there is a cofinal subset $A_{0}\subseteq A$
such that $f$ is constant on $A_{0}$.\end{lem}
\begin{proof}
Assume not, then for every $\alpha<\kappa$ there is some $a_{\alpha}\in A$
such that $f(a)\neq\alpha$ for any $a\geq a_{\alpha}$. By $\kappa^{+}$-directedness
there is some $a\geq a_{\alpha}$ for all $\alpha<\kappa$. But then
whatever $f(a)$ is, we get a contradiction. \end{proof}
\begin{lem}
\label{lem: Countable base} Assume that $p(x)\in S(A)$ does not
fork over $B$. Then there is some $B_{0}\subseteq B$ such that $|B_{0}|\leq|A|+|T|$
and $p(x)$ does not fork over $B_{0}$.\end{lem}
\begin{proof}
Let $\kappa=|A|+|T|$, and assume the converse. Then $p\left(x\right)$
forks over every $C\subseteq B$ with $\left|C\right|\leq\kappa$.
That is, for every $C\in B^{\left[\kappa\right]}$ there are $p_{C}\subseteq p$
with $|p_{C}|<\omega$, $\psi_{0}^{C}\left(x,y_{0}\right),\ldots,\psi_{m_{C}-1}^{C}\left(x,y_{m_{C}}\right)\in L$
and $k_{C}<\omega$ such that for some $d_{0}^{C},...,d_{m_{C}-1}^{C}$,
$p_{C}\left(x\right)\vdash\bigvee_{i<m_{C}}\psi_{i}^{C}\left(x,d_{i}^{C}\right)$
and each of $\psi_{i}^{C}\left(x,d_{i}^{C}\right)$ is $k_{C}$-dividing
over $C$. As $B^{\left[\kappa\right]}$ is $\kappa^{+}$-directed
under inclusion and $\left|p\left(x\right)\right|\leq\kappa$, it
follows by Lemma \ref{lem: Cofinal set} that for some finite $p_{0}\subseteq p$,
$\left\{ \psi_{i}\left|\,i<m\right.\right\} $ and $k$ this holds
for every $C\in B^{\left[\kappa\right]}$. But then by compactness
$p_{0}(x)$ forks over $B$ --- a contradiction.
\end{proof}

\subsection{The non-forking spectra}
\begin{defn}

\begin{enumerate}
\item For a countable first-order $T$ and infinite cardinals $\kappa\leq\lambda$,
let
\[
f_{T}\left(\kappa,\lambda\right)=\mbox{sup}\left\{ S^{\nf}(N,M)\left|\,M\preceq N\models T,\,|M|\leq\kappa,\,|N|\leq\lambda\right.\right\} ,
\]
where $S^{\mbox{\ensuremath{\nf}}}\left(A,B\right)=\left\{ p\in S_{1}(A)\left|\,p\mbox{ does not fork over }B\right.\right\} $.
We call this function the \emph{non-forking spectrum} of $T$.
\item For $n>1$, we may also define $f_{T}^{n}\left(\kappa,\lambda\right)$
and $S_{n}^{\nf}$ similarly where we replace $1$-types with $n$-types.
\end{enumerate}
\end{defn}
\begin{note}
All the proofs in Section \ref{sec:Gaps} remain valid for $f_{T}$
replaced by $f_{T}^{n}$. \end{note}
\begin{rem}
A special case $f_{T}(\kappa,\kappa)$ is the well-known stability
function $f_{T}(\kappa)$ because $S^{\nf}\left(N,N\right)=S\left(N\right)$
 (Because every type over a model $M$ does not fork over $M$).
\end{rem}
Some easy observations:
\begin{lem}
\label{lem:basic properties of f_T}For all $\kappa\leq\lambda$,
\begin{enumerate}
\item $f_{T}\left(\kappa\right)\leq f_{T}\left(\kappa,\lambda\right)$
\item $\kappa\leq f_{T}\left(\kappa,\lambda\right)\leq2^{\lambda}$
\item If $f_{T}\left(\kappa,\lambda\right)\geq\mu$ and $\kappa\leq\kappa'$
then $f_{T}\left(\kappa',\lambda\right)\geq\mu$.
\item $f_{T}^{n}\left(\kappa,\lambda\right)\leq f_{T}^{n+1}\left(\kappa,\lambda\right)$
\end{enumerate}
\end{lem}
For set theoretic preliminaries, see Section \ref{sec:on ded kappa < ded kappa ^aleph0}.

\section{\label{sec:Gaps}Gaps}

In the following series of subsections, we exclude all the possibilities
for $f_{T}$ which are not in our list (except when ``???'' is indicated).

\subsection{On (\ref{enu:kappa}) -- (\ref{enu:ded kappa}).}
\begin{defn}
\label{def:stable theories}Recall that a theory $T$ is called \emph{stable}
if $f_{T}\left(\kappa\right)\leq\kappa^{\aleph_{0}}$ for all $\kappa$
(see \cite[Theorem II.2.13]{Sh:c} for equivalent definitions). \end{defn}
\begin{rem}
\label{rem:stable theories}If $T$ is stable then every type over
a model $M$ has a unique non-forking extension to any model containing
$M$, so $f_{T}\left(\kappa\right)=f_{T}\left(\kappa,\lambda\right)$
for all $\lambda\geq\kappa\geq\aleph_{0}$.

If $T$ is unstable, then $f_{T}\left(\kappa\right)\geq\ded\left(\kappa\right)$
for all $\kappa$ (see \cite[Theorem II.2.49]{Sh:c}), so $f_{T}\left(\kappa,\lambda\right)\geq\ded\left(\kappa\right)$
for all $\lambda\geq\kappa$. \end{rem}
\begin{prop}
The following holds:
\begin{enumerate}
\item If $f_{T}\left(\kappa,\lambda\right)>\kappa$ for some $\lambda\geq\kappa$
then $f_{T}\left(\kappa,\lambda\right)\geq\kappa+2^{\aleph_{0}}$
for all $\lambda\geq\kappa$.
\item If $f_{T}\left(\kappa,\lambda\right)>\kappa+2^{\aleph_{0}}$ for some
$\lambda\geq\kappa$ then $f_{T}\left(\kappa,\lambda\right)\geq\kappa^{\aleph_{0}}$
for all $\lambda\geq\kappa$.
\item If $f_{T}\left(\kappa,\lambda\right)>\kappa^{\aleph_{0}}$ for some
$\lambda\geq\kappa$ then $f_{T}\left(\kappa,\lambda\right)\geq\ded\left(\kappa\right)$
for all $\lambda\geq\kappa$.
\end{enumerate}
\end{prop}
\begin{proof}
(3): Suppose $f_{T}\left(\kappa,\lambda\right)>\kappa^{\aleph_{0}}$
for some $\lambda\geq\kappa$. Then $T$ is unstable, so by Remark
\ref{rem:stable theories} $f_{T}\left(\kappa,\lambda\right)\geq\ded\left(\kappa\right)$
for all $\lambda\geq\kappa$.

(1): Suppose $f_{T}\left(\kappa,\lambda\right)>\kappa$ for some $\lambda\geq\kappa$.
Without loss of generality $T$ is stable. So $f_{T}\left(\kappa\right)=f_{T}\left(\kappa,\lambda\right)>\kappa$.
By Fact \ref{fac:Keisler}, $f_{T}\left(\kappa\right)\geq\kappa+2^{\aleph_{0}}$
for all $\kappa$, and we are done.

(2): Similar to (1).
\end{proof}

\subsection{The gap between (\ref{enu:ded kappa ^ omega}) and (\ref{enu:double exponent kappa}). }
\begin{defn}

\begin{enumerate}
\item A formula $\varphi\left(x,y\right)$ has the\emph{ independence property}
(IP) if there are \\
$\left\{ a_{i}\left|\,i<\omega\right.\right\} $ and $\left\{ b_{s}\left|\,s\subseteq\omega\right.\right\} $
in $\C$ such that $\varphi\left(a_{i},b_{s}\right)$ holds if and
only if $i\in s$ for all $i<\omega$ and $s\subseteq\omega$.
\item A theory $T$ is \emph{NIP (dependent)} if no formula $\varphi\left(x,y\right)$
has IP.
\end{enumerate}
\end{defn}
See \cite{Ad} for more about NIP.
\begin{fact}
\label{fac:NIP}If $T$ is NIP and $M\models T$ then the $\left|S\left(M\right)\right|\leq\left(\ded\left|M\right|\right)^{\aleph_{0}}$
\cite{shelahStability} and if $M\prec N$ and $p\in S\left(M\right)$
then $p$ has at most $\left(\ded\left|M\right|\right)^{\aleph_{0}}$
non-forking extensions (e.g. follows from the proof of \cite[Theorem 42]{Ad},
noticing that $\left|S_{\omega}\left(M\right)\right|\leq\left(\ded\left|M\right|\right)^{\aleph_{0}}$).
It follows that $\left|S^{\nf}\left(N,M\right)\right|\leq\left(\ded\left|M\right|\right)^{\aleph_{0}}$.
\end{fact}
A generalization of a result due to Poizat \cite{poiThUnstable}.
\begin{prop}
\label{prop: many ultrafilters from IP} Assume that $f_{T}\left(\kappa,\lambda\right)>\left(\ded\kappa\right)^{\aleph_{0}}$
for some $\lambda\geq\kappa$. Then $f_{T}(\kappa,\lambda)\geq2^{\min\{\lambda,2^{\kappa}\}}$
for all $\lambda\geq\kappa$.\end{prop}
\begin{proof}
By Fact \ref{fac:NIP}, some formula $\varphi\left(x,y\right)$ in
$T$ has IP.

Recall that a set $S\subseteq\SS\left(\kappa\right)$ is called independent
if every finite intersection of elements of $S$ or their complements
is non-empty. By a theorem of Hausdorff there is such a family of
size $2^{\kappa}$. Fix some $\kappa$ and $\mu\leq2^{\kappa}$, and
let $S$ be a family of independent subset of $\kappa$, such that
$|S|=\mu$.

Let $A=\left\{ a_{i}\left|\,i<\kappa\right.\right\} $ be such that
$b_{s}\models\left\{ \varphi\left(x,a_{i}\right)^{\mbox{if }i\in s}\left|\,i<\kappa\right.\right\} $
for every $s\subseteq\kappa$. Let $M$ be a model of size $\kappa$
containing $A$ and $N$ of size $\mu$ containing $M\cup\left\{ b_{s}\left|\,s\in S\right.\right\} $.
Now for every $D\subseteq S$, there is an ultrafilter on $\kappa$
containing $D$, and let $p_{D}\in S\left(N\right)$ be
\[
\left\{ \psi\left(x,c\right)\left|\,c\in N,\,\psi\in L,\,\left\{ a\in M\left|\,\psi\left(a,c\right)\right.\right\} \in D\right.\right\} ,
\]
so it is finitely satisfiable in $A$. Notice that if $D_{1}\neq D_{2}$
then $p_{D_{1}}\neq p_{D_{2}}$, as $\varphi\left(x,b_{s}\right)\in p_{D_{1}}\land\neg\varphi\left(x,b_{s}\right)\in p_{D_{2}}$
for any $s\in D_{1}\setminus D_{2}$. Thus $S^{\nf}\left(N,M\right)\geq2^{\mu}$.

If $\lambda\leq2^{\kappa}$, then let $\mu=\lambda$ and we have that
$f_{T}\left(\lambda,\kappa\right)\geq2^{\lambda}$.

If $\lambda>2^{\kappa}$, then let $\mu=2^{\kappa}$, so $f_{T}\left(\kappa,\lambda\right)\geq2^{2^{\kappa}}$
and we are done.
\end{proof}
Note that in the Main Theorem we assumed that $\lambda\geq2^{2^{\kappa}}$,
so in this case we have $f_{T}\left(\kappa,\lambda\right)\geq2^{2^{\kappa}}$.

\subsection{The gap between (\ref{enu:double exponent kappa}) and (\ref{enu:lambda}).\protect \\
}

We recall the basic properties of splitting.
\begin{defn}
Suppose $A\subseteq B$ are sets. A type $p\left(x\right)\in S\left(B\right)$
\emph{splits} over $A$ if there is some formula $\varphi\left(x,y\right)$
 and $b,\,c\in B$ such that $\tp\left(b/A\right)=\tp\left(c/A\right)$
and $\varphi\left(x,b\right)\land\neg\varphi\left(x,c\right)\in p$.
\end{defn}

\begin{fact}
\label{fac:non-splitting}(See e.g. \cite[Sections 5, 6]{Ad}) Let
$M\prec N$ be models
\begin{enumerate}
\item The number of types in $S\left(N\right)$ that do not split over $M$
is bounded by $2^{2^{\left|M\right|}}$.
\item If $N$ is $\left|M\right|^{+}$-saturated and $p\in S\left(N\right)$
splits over $M$, then there is an indiscernible sequence $\left\langle a_{i}\left|\,i<\omega\right.\right\rangle $
in $N$ over $M$ such that $\varphi\left(x,a_{0}\right)\land\neg\varphi\left(x,a_{1}\right)\in p$
for some $\varphi$.
\item If $T$ is NIP, and $p\in S^{\nf}\left(N,M\right)$, then $p$ does
not split over $M$.
\end{enumerate}
\end{fact}
\begin{defn}
\label{def:non-forking-pattern}A \emph{non-forking pattern} of depth
$\theta$ over $A$ consists of an array $\left\{ \bar{a}_{\alpha}\left|\,\alpha<\theta\right.\right\} $
where $\bar{a}_{\alpha}=\left\langle a_{\alpha,i}\left|\,i<\omega\right.\right\rangle $
and formulas $\left\{ \varphi_{\alpha}\left(x,y\right)\left|\,\alpha<\theta\right.\right\} $
such that
\begin{itemize}
\item $\bar{a}_{\alpha_{0}}$ is indiscernible over $\left\{ \bar{a}_{\alpha}\left|\,\alpha<\alpha_{0}\right.\right\} \cup A$.
\item $\left\{ \varphi_{\alpha}\left(x,a_{\alpha,0}\right)\land\neg\varphi_{\alpha}\left(x,a_{\alpha,1}\right)\left|\,\alpha<\theta\right.\right\} $
does not fork over $A$.
\end{itemize}
\end{defn}

\begin{defn}
A \emph{pair non-forking pattern} of depth $\theta$ over a set $A$
is defined similarly, but here we only demand that $\bar{a}_{\alpha_{0}}$
is indiscernible over $\left\{ a_{\alpha,0},a_{\alpha,1}\left|\,\alpha<\alpha_{0}\right.\right\} \cup A$.\end{defn}
\begin{lem}
\label{lem:pair pattern =00003D full pattern}If there is a pair non-forking
pattern of depth $\theta$ over $A$, then there is a non-forking
pattern of depth $\theta$ over $A$.\end{lem}
\begin{proof}
Suppose we have a pair non-forking pattern of depth $\theta$, $\left\{ \bar{a}_{\alpha}\left|\,\alpha<\theta\right.\right\} $.
It is enough to find an array $\left\{ \bar{b}_{\alpha}\left|\,\alpha<\theta\right.\right\} $
as in the first point of Definition \ref{def:non-forking-pattern}
such that $b_{\alpha,0}b_{\alpha,1}=a_{\alpha,0}a_{\alpha,1}$. By
compactness we may assume that $\theta$ is finite. The proof is by
induction on $\theta$. For $\theta=0,1$ there is nothing to do.
Suppose $\theta=n+1$. By induction, we may assume that the first
$n$ sequences satisfy the first point. By Ramsey and compactness
(see e.g. \cite[Lemma 5.1.3]{TentZiegler}), there is an indiscernible
sequence $\bar{b}_{n}'$ which is indiscernible over $A\cup\left\{ \bar{a}_{\alpha}\left|\,\alpha<n\right.\right\} $
and such that the type of any finite sub-tuple in $\bar{b}_{n}'$
is the same as a sub-tuple of the same length in $\bar{a}_{n}$ over
$A\cup\left\{ a_{\alpha,0},a_{\alpha,1}\left|\,\alpha<n\right.\right\} $.
So there is an automorphism taking $\bar{b}_{n}'$ to $\bar{a}_{n}$
which fixes $A\cup\left\{ a_{\alpha,0},a_{\alpha,1}\left|\,\alpha<n\right.\right\} $.
Now let $\bar{b}_{\alpha}$ for $\alpha<n$ be the image of this automorphism,
and $\bar{b}_{n}=\bar{a}_{n}$. \end{proof}
\begin{defn}
For an infinite cardinal $\kappa$, let $g_{T}\left(\kappa\right)$
be the smallest cardinal $\theta$ such that there is no (pair) non-forking
pattern of depth $\theta$ over some model of size $\kappa$. \end{defn}
\begin{rem}
\label{rem:reducing base of g}It is clear that $g_{T}\left(\kappa'\right)\geq g_{T}\left(\kappa\right)$
whenever $\kappa'\geq\kappa$. In addition, from Lemma \ref{lem: Countable base}
it follows that if $g_{T}\left(\kappa\right)>\theta$ then $g_{T}\left(\theta+\aleph_{0}\right)>\theta$. \end{rem}
\begin{lem}
\label{lem:better pattern}If $g_{T}\left(\kappa\right)>\theta$ then
there is $M$ of size $\kappa$ such that for any $\lambda$ we can
find a non-forking pattern $\left\{ \bar{a}_{\alpha},\,\varphi_{\alpha}\left|\,\alpha<\theta\right.\right\} $
such that in addition:
\begin{itemize}
\item $\bar{a}_{\alpha}=\left\langle a_{\alpha,i}\left|\,i<\lambda\right.\right\rangle $
\item $\left\{ \varphi_{\alpha}\left(x,a_{\alpha,0}\right)\left|\,\alpha<\theta\right.\right\} \cup\left\{ \neg\varphi_{\alpha}\left(x,a_{\alpha,i}\right)\left|\,\alpha<\theta,\,0<i<\lambda\right.\right\} $
does not fork over $M$.
\end{itemize}
\end{lem}
\begin{proof}
By assumption we have some non-forking pattern $\left\{ \bar{a}_{\alpha},\varphi_{\alpha}\left|\,\alpha<\theta\right.\right\} $
over some $M$ of size $\kappa$. By compactness, we may assume that
$\bar{a}_{\alpha}$ is of length $\lambda$ for all $\alpha<\theta$.
Let $p\left(x\right)\in S\left(\C\right)$ be a non-forking extension
of $\left\{ \varphi_{\alpha}\left(x,a_{\alpha,0}\right)\land\neg\varphi_{\alpha}\left(x,a_{\alpha,1}\right)\left|\,\alpha<\theta\right.\right\} $.
By omitting some elements from each sequence $\bar{a}_{\alpha}$ and
maybe changing $\varphi_{\alpha}$ to $\neg\varphi_{\alpha}$ we may
assume
\[
\left\{ \varphi_{\alpha}\left(x,a_{\alpha,0}\right)\left|\,\alpha<\theta\right.\right\} \cup\left\{ \neg\varphi_{\alpha}\left(x,a_{\alpha,i}\right)\left|\,\alpha<\theta,\,0<i<\lambda\right.\right\} \subseteq p.
\]
\end{proof}
\begin{prop}
The following are equivalent:
\begin{enumerate}
\item For some $\kappa$, $g_{T}\left(\kappa\right)>1$.
\item For every $\lambda\geq\kappa\geq\aleph_{0}$, $f_{T}(\kappa,\lambda)=2^{\lambda}$
if $\lambda\leq2^{\kappa}$ and $f_{T}(\kappa,\lambda)\geq\lambda$
otherwise.
\item For some $\lambda\geq\kappa$, $f_{T}(\kappa,\lambda)>2^{2^{\kappa}}$.
\end{enumerate}
\end{prop}
\begin{proof}
(1) implies (2): By remark \ref{rem:reducing base of g}, we may assume
that $\kappa=\aleph_{0}$. By Lemma \ref{lem:better pattern} there
is some countable $M$ such that for any $\lambda$ there is some
$\bar{b}=\left\langle b_{i}\left|\,i<\lambda\right.\right\rangle $
such that $\left\{ \varphi\left(x,b_{0}\right)\right\} \cup\left\{ \neg\varphi\left(x,b_{i}\right)\left|\,i<\lambda\right.\right\} $
does not fork over $M$. So, for every $i<\lambda$, $p_{i}\left(x\right)=\left\{ \varphi\left(x,b_{j}\right)^{\mbox{if }j=i}\left|\,i\leq j<\lambda\right.\right\} $
does not fork over $M$.

Taking some model $N\supseteq\bar{b}$ of size $\lambda$ we can expand
each $p_{i}$ to some $q_{i}\in S^{\nf}\left(N,M\right)$. Notice
that for any $i<j<\lambda$, $q_{i}\neq q_{j}$ as $\neg\varphi\left(x,a_{j}\right)\in p_{i}$,
but $\varphi\left(x,a_{j}\right)\in p_{j}$. So we conclude that $S^{\nf}\left(N,M\right)\geq\lambda$.
By Lemma \ref{lem:basic properties of f_T}, we get that $f_{T}\left(\kappa,\lambda\right)\geq\lambda$
for every $\lambda\geq\kappa$.

Note that by Fact \ref{fac:NIP}, we know that $T$ is not NIP, so
if $\lambda\leq2^{\kappa}$, then by Proposition \ref{prop: many ultrafilters from IP}
$f_{T}\left(\kappa,\lambda\right)=2^{\lambda}$.

(2) implies (3) is clear.

(3) implies (1): Let $M\prec N$ witness that $f_{T}(\kappa,\lambda)>2^{2^{\kappa}}$.
By Fact \ref{fac:non-splitting}(1), there is some $p\in S^{\nf}\left(N,M\right)$
that splits over $M$.

Let $N'\succ N$ be $|M|^{+}$-saturated and $p'\in S^{\nf}\left(N',M\right)$,
a non-forking extension of $p$. By Fact \ref{fac:non-splitting}(2)
we find an indiscernible sequence $\bar{a}=\left\langle a_{i}\left|\,i<\omega\right.\right\rangle $
in $N'$ and a formula $\varphi\left(x,a_{0}\right)\land\neg\varphi\left(x,a_{1}\right)\in p$
--- and we get (1).
\end{proof}

\subsection{The gap between (\ref{enu:lambda}) and (\ref{enu:lambda ^omega}).}
\begin{lem}
\label{lem:cardinal arithmatics} For any cardinals $\lambda$ and
$\theta$, if $\theta$ is regular or $\lambda\geq2^{<\theta}$ then
$\left(\lambda^{<\theta}\right)^{<\theta}=\lambda^{<\theta}$.\end{lem}
\begin{proof}
By \cite[Observation 2.11 (4)]{Sh233}, if $\lambda\geq2^{<\theta}$,
then $\lambda^{<\theta}=\lambda^{\nu}$ for some $\nu<\theta$. So
$\left(\lambda^{<\theta}\right)^{<\theta}=\left(\lambda^{\nu}\right)^{<\theta}=\lambda^{<\theta}$.
If $\theta$ is regular, then, letting $\lambda'=\lambda^{<\theta}$,
since $\lambda'\geq2^{<\theta}$, $\left(\lambda'\right)^{<\theta}=\left(\lambda'\right)^{\nu}$
for some $\nu<\theta$ so
\[
\left(\lambda'\right)^{<\theta}=\left(\lambda'\right)^{\nu}=\left(\lambda^{<\theta}\right)^{\nu}=\left(\sum_{\mu<\theta}\lambda^{\mu}\right)^{\nu}=\sum_{\mu<\theta}\left(\lambda^{\mu\cdot\nu}\right)=\lambda^{<\theta}=\lambda'.
\]
\end{proof}
\begin{lem}
\label{lem:polynomial f_T implies g_T bigger than degree}Suppose
$f_{T}\left(\kappa,\lambda\right)>\lambda^{<\theta}$, and $\lambda\geq\sum_{\mu<\theta}2^{2^{\kappa+\mu}}$
then $g_{T}\left(\kappa\right)>\theta$.\end{lem}
\begin{proof}
Let $\lambda'=\lambda^{<\theta}$. By Lemma \ref{lem:cardinal arithmatics},
$\left(\lambda'\right)^{<\theta}=\lambda'$. So, we have $f_{T}\left(\kappa,\lambda'\right)\geq f_{T}\left(\kappa,\lambda\right)>\lambda^{<\theta}=\left(\lambda'\right)^{<\theta}$,
so we may replace $\lambda$ with $\lambda'$ and assume $\lambda^{<\theta}=\lambda$.

Let $\left(N,M\right)$ be a witness to $f_{T}\left(\kappa,\lambda\right)>\lambda$.
For every $A\subseteq N$ of size $<\theta$, let $M_{A}\subseteq\C$
be a $\left(\kappa+\left|A\right|\right)^{+}$-saturated model of
size $\leq2^{\left|A\right|+\kappa}$ containing $M\cup A$. Let $N_{0}=\bigcup_{A\in N^{\left[<\theta\right]}}M_{A}$.
So $N_{0}\supseteq N$, and $\left|N_{0}\right|\leq\lambda\cdot2^{<\theta+\kappa}=\lambda$.
Repeating the construction with respect to $\left(N_{0},M\right)$,
construct $N_{1}$, and more generally $N_{i}$ for $i\leq\theta$,
taking union in limit steps. So $\left|N_{\theta}\right|\leq\lambda\cdot\theta=\lambda$.

Fix $p\left(x\right)\in S^{\nf}\left(N_{\theta},M\right)$.

 We try
to choose by induction on $\alpha<\theta$ formulas $\varphi_{\alpha}^{p}\left(x,y\right)$
and sequences $\bar{a}_{\alpha}^{p}=\left\langle a_{\alpha,i}^{p}\left|\,i<\omega\right.\right\rangle $
in $N_{\alpha+1}$ such that $\bar{a}_{\alpha}^{p}$ is indiscernible
over $\left\{ a_{\beta,0}^{p},a_{\beta,1}^{p}\left|\,\beta<\alpha\right.\right\} \cup M$
and $\varphi_{\alpha}^{p}\left(x,a_{\alpha,0}^{p}\right)\land\neg\varphi_{\alpha}^{p}\left(x,a_{\alpha,1}^{p}\right)\in p$.
If we succeed, then we found a pair-non-forking pattern of depth $\theta$
over $M$ as desired (by Lemma \ref{lem:pair pattern =00003D full pattern}).
Otherwise, we are stuck in some $\alpha_{p}<\theta$. Let $A_{p}=\bigcup\left\{ a_{\beta,0}^{p},a_{\beta,1}^{p}\left|\,\beta<\alpha_{p}\right.\right\} $.

Let $F\subseteq S^{\nf}\left(N_{\theta},M\right)$ be a set of size
$>\lambda$ such that for $p\neq q\in F$, $p|_{N}\neq q|_{N}$. As
the size of the set $\left\{ A_{p}\left|\,p\in F\right.\right\} $
is bounded by $\lambda^{<\theta}=\lambda$ there is some $A$ of size
$<$ $\theta$ and $\alpha$ such that, letting $S=\left\{ p\in F\left|\,A_{p}=A\land\alpha_{p}=\alpha\right.\right\} $,
$\left|S\right|>\lambda$. Let $M_{0}\subseteq N_{\alpha}$ be some
model containing $A\cup M$ of size $\kappa+\left|A\right|$. Suppose
$p\in S$ and $p|_{N_{\alpha}}$ splits over $M_{0}$, so already
$p|_{M_{0}B}$ splits over $M_{0}$ for some finite $B$. Then there
is some $\left(\kappa+\left|A\right|\right)^{+}$-saturated model
$N'\subseteq N_{\alpha+1}$ containing $M\cup A\cup B$ and some $M_{0}'\subseteq N'$
such that $M_{0}'\equiv_{MAB}M_{0}$, so $p|_{N'}$ splits over $M_{0}'$.
By Fact \ref{fac:non-splitting}(2), we can find an $M_{0}'$-indiscernible
sequence $\left\langle a_{\alpha,i}^{p}\left|\,i<\omega\right.\right\rangle $
in $N'\subseteq N_{\alpha+1}$ such that $\varphi\left(x,a_{\alpha,0}^{p}\right)\land\neg\varphi\left(x,a_{\alpha,1}^{p}\right)\in p$
--- contradicting the choice of $\alpha$. So, for every $p\in S$,
$p|_{N_{\alpha}}$ does not split over $M_{0}$. But then by the choice
of $F$ and Fact \ref{fac:non-splitting}(1), $\left|S\right|\leq2^{2^{\kappa+\left|A\right|}}$
--- contradiction. \end{proof}
\begin{lem}
\label{lem:If g is bigger than theta then f is bigger than lambda^theta}If
$g_{T}\left(\kappa\right)>\theta$ then $f_{T}\left(\kappa,\lambda\right)\geq\lambda^{\left\langle \theta\right\rangle _{\tr}}$
for all $\lambda\geq\kappa$ (see Definition \ref{def:trees}). \end{lem}
\begin{proof}
Fix $\lambda\geq\kappa+\theta$ (if $\lambda<\theta$ then $\lambda^{\left\langle \theta\right\rangle _{\tr}}$
is $0$). By Lemma \ref{lem:better pattern}, there is some non-forking
pattern $\left\{ \bar{a}_{\alpha},\,\varphi_{\alpha}\left|\,\alpha<\theta\right.\right\} $
over a model $M$ of size $\kappa$ such that $\bar{a}_{\alpha}=\left\langle a_{\alpha,i}\left|\,i<\lambda\right.\right\rangle $
and $p\left(x\right)=\left\{ \varphi_{\alpha}\left(x,a_{\alpha,0}\right)\left|\,\alpha<\theta\right.\right\} \cup\left\{ \neg\varphi_{\alpha}\left(x,a_{\alpha,i}\right)\left|\,\alpha<\theta,\,0<i<\lambda\right.\right\} $
does not fork over $M$. By induction on $\beta\leq\theta$ we define
elementary mappings $F_{\eta}$, $\eta\in\lambda^{\beta}$, with $\mbox{dom}(F_{\eta})=A_{\beta}=M\cup\left\{ \bar{a}_{\alpha}\left|\,\alpha<\beta\right.\right\} $:
\begin{itemize}
\item $F_{\emptyset}$ is the identity on $M$.
\item If $\beta$ is a limit ordinal, then let $F_{\eta}=\bigcup_{\alpha<\beta}F_{\eta\restriction\alpha}$.
\item If $\beta=\alpha+1$, let $F_{\eta0}$ be an arbitrary extension of
$F_{\eta}$ to $A_{\alpha+1}$. For $i<\lambda$, let $F_{\eta i}$
be an arbitrary elementary mapping extending $F_{\eta}$ such that
$F_{\eta i}\left(a_{\alpha,j}\right)=F_{\eta0}\left(a_{\alpha,i+j}\right)$.
This could be done by indiscerniblity.
\end{itemize}
Let $p_{\eta}=F_{\eta}\left(p\right)$. So,
\begin{itemize}
\item $p_{\eta}\left(x\right)$ does not fork over $M$ --- as $F_{\eta}$
is an elementary map fixing $M$.
\item If $\eta\neq\nu\in\lambda^{\theta}$, then $p_{\eta}\neq p_{\nu}$.
To see it, let $\alpha=\min\left\{ \beta<\theta\left|\,\eta\restriction\beta\neq\nu\restriction\beta\right.\right\} $
and suppose $\alpha=\beta+1$, $\rho=\eta\upharpoonright\beta=\nu\upharpoonright\beta$.
Assume $\eta\left(\beta\right)=i<j=\nu\left(\beta\right)$ and $0<k<\lambda$
is such that $i+k=j$. Then $\varphi\left(x,a_{\alpha,0}\right)\in p\Rightarrow\varphi\left(x,F_{\nu}\left(a_{\alpha,0}\right)\right)\in p_{\nu}$.
Similarly, $\neg\varphi\left(x,a_{\alpha,k}\right)\in p\Rightarrow\neg\varphi\left(x,F_{\eta}\left(a_{\alpha,k}\right)\right)\in p_{\eta}$.
But,
\[
F_{\nu}\left(a_{\alpha,0}\right)=F_{\rho j}\left(a_{\alpha,0}\right)=F_{\rho0}\left(a_{\alpha,j}\right)=F_{\rho0}\left(a_{i+k}\right)=F_{\rho i}\left(a_{\alpha,k}\right)=F_{\eta}\left(a_{\alpha,k}\right),
\]
so $p_{\eta}\neq p_{\nu}$.
\end{itemize}
Let $T\subseteq\lambda^{<\theta}$ be a tree of size $\leq\lambda$
such that if $x\in T$ and $y<x$ then $y\in T$. Let $B=\bigcup\left\{ F_{\eta}\left(\bar{a}_{\alpha}\right)\left|\,\alpha<\lg\left(\eta\right)\land\eta\in T\right.\right\} \cup M$,
so $\left|B\right|\leq\lambda+\kappa=\lambda$. Let $N$ be some model
containing $B$ of size $\lambda$. Thus, $\left|S^{\nf}\left(N,M\right)\right|$
is at least the number of branches in $T$ of length $\theta$. By
the definition of $\lambda^{\left\langle \theta\right\rangle _{\tr}}$
we are done. \end{proof}
\begin{prop}
If $f_{T}\left(\kappa,\lambda\right)>\lambda$ for some $\lambda\geq2^{2^{\kappa}}$,
then $f_{T}\left(\kappa,\lambda\right)\geq\lambda^{\aleph_{0}}$ for
all $\lambda\geq\kappa$.\end{prop}
\begin{proof}
By Lemma \ref{lem:polynomial f_T implies g_T bigger than degree},
taking $\theta=\aleph_{0}$, $g_{T}\left(\kappa\right)>\aleph_{0}$
and then by Remark \ref{rem:reducing base of g}, $g_{T}\left(\aleph_{0}\right)>\aleph_{0}$.
By Lemma \ref{lem:If g is bigger than theta then f is bigger than lambda^theta},
$f_{T}\left(\aleph_{0},\lambda\right)>\lambda^{\left\langle \aleph_{0}\right\rangle }$
for all $\lambda$ but $\lambda^{\left\langle \aleph_{0}\right\rangle }=\lambda^{\aleph_{0}}$
(see Remark \ref{rem:treePower=00003Dpower}). By Remark \ref{lem:basic properties of f_T},
$f_{T}\left(\kappa,\lambda\right)\geq f_{T}\left(\aleph_{0},\lambda\right)\geq\lambda^{\aleph_{0}}$
so we are done.
\end{proof}

\subsection{On (\ref{enu:??? between lambda^omega and lambda^<beth}).}
\begin{prop}
If $f_{T}\left(\kappa,\lambda\right)>\lambda^{\mu}$ for some $\lambda\geq2^{2^{\kappa+\mu}}$,
then $f_{T}\left(\kappa,\lambda\right)\geq\lambda^{\left\langle \mu^{+}\right\rangle _{\tr}}$
for all $\lambda\geq\kappa\geq\mu^{+}$. \end{prop}
\begin{proof}
By Lemma \ref{lem:polynomial f_T implies g_T bigger than degree},
$g_{T}\left(\kappa\right)>\mu^{+}$. By Lemma \ref{lem: Countable base},
$g_{T}\left(\mu^{+}\right)>\mu^{+}$. By Lemma \ref{lem:If g is bigger than theta then f is bigger than lambda^theta},
$f_{T}\left(\mu^{+},\lambda\right)\geq\lambda^{\left\langle \mu^{+}\right\rangle _{\tr}}$
for all $\lambda\geq\mu^{+}$, and so by Lemma \ref{lem:basic properties of f_T}
, $f_{T}\left(\kappa,\lambda\right)\geq\lambda^{\left\langle \mu^{+}\right\rangle _{\tr}}$
for any $\lambda\geq\kappa\geq\mu^{+}$. \end{proof}
\begin{cor}
If $f_{T}\left(\kappa,\lambda\right)>\lambda^{\aleph_{n}}$ for some
$\lambda\geq2^{2^{\kappa+\aleph_{n}}}$, then $f_{T}\left(\kappa,\lambda\right)\geq\lambda^{\left\langle \aleph_{n+1}\right\rangle _{\tr}}$
for all $\lambda\geq\kappa\geq\aleph_{n+1}$.
\end{cor}
This corollary says that morally there are gaps between $\lambda$
and $\lambda^{\aleph_{0}}$, $\lambda^{\aleph_{0}}$ and $\lambda^{\aleph_{1}}$
etc.

\subsection{On the gap between (\ref{enu:lambda^beth}) and (\ref{enu:ded lambda}).\protect \\
}

The following fact follows from the proof of Morley's two cardinal
theorem. For details, see \cite[Theorem 23]{KeislerInf}.
\begin{fact}
\label{fac:L_omega1 sentence} Suppose $\psi\in L_{\omega_{1},\omega}$,
$<$ is a binary relation, $P$ and $Q$ are predicates in $L$ and
$\psi$ implies that ``$<$ is a linear order on $Q$''. If for
every countable ordinal $\varepsilon$ there is a structure $B$ such
that
\begin{itemize}
\item $B\models\psi$
\item There is an embedding of the order $\beth_{\varepsilon}\left(\left|P^{B}\right|\right)$
into $\left(Q^{B},<^{B}\right)$.
\end{itemize}
Then for every cardinal $\lambda$ there is some structure $B$ such
that
\begin{itemize}
\item $B\models\psi$
\item $\left|P^{B}\right|=\aleph_{0}$
\item there is an embedding of $\left(\lambda,<\right)$ into $\left(Q^{B},<^{B}\right)$.
\end{itemize}
\end{fact}
\begin{lem}
\label{lem:forking is L_omega1 definable}Let $M\prec N$ and $a\in N$.
Then the following are equivalent:
\begin{enumerate}
\item $\varphi\left(x,a\right)$ forks over $M$.
\item The following holds in $N$:
\begin{eqnarray*}
 & \bigvee_{\left\{ \psi_{0},\ldots,\psi_{m-1}\right\} \subseteq L}\bigvee_{k_{i}<\omega,i<m}\bigwedge_{\Delta\subseteq L\,\mbox{finite}}\\
 &
 \bigwedge_{n<\omega}\forall c_{0},\ldots,c_{n-1}\in M\exists\bar{y}_{0},\ldots,\exists\bar{y}_{m-1}\\
 & \left[\varphi\left(x,a\right)\vdash\bigvee_{i<n}\psi\left(x,y_{i,0}\right)\land \right. \\
 &
 \left. \bigwedge_{i<m,j<n}\left(y_{i,j}\equiv_{\bar{c}}^{\Delta}y_{i,0}\right)\land\bigwedge_{i<m,\,s\in n^{\left[k_{i}\right]}}\forall x\left(\neg\bigwedge_{j\in s}\varphi\left(x,y_{i,j}\right)\right)\right]
\end{eqnarray*}
 where $\bar{y}_{i}=\left\langle y_{i,j}\left|\,j<n\right.\right\rangle $
for $i<m$ and $\bar{c}=\left\langle c_{i}\left|\,i<n\right.\right\rangle $.
\end{enumerate}
\end{lem}
\begin{proof}
By compactness. \end{proof}
\begin{lem}
\label{lem:big g implies formula is constant}If $g_{T}\left(\kappa\right)>\mu>\aleph_{0}$,
then there is a non-forking pattern $\left\{ \varphi_{\alpha},\bar{a}_{\alpha}\left|\,\alpha<\mu\right.\right\} $
such that $\varphi_{\alpha}=\varphi$ for some formula $\varphi$.\end{lem}
\begin{proof}
By pigeon-hole.\end{proof}
\begin{prop}
\label{prop:If g is bigger than beth omega1 then it is infinite}If
for all $\varepsilon<\aleph_{1}$, there is some $\kappa$ such that
$g_{T}\left(\kappa\right)>\beth_{\varepsilon}\left(\kappa\right)$
then $g_{T}\left(\aleph_{0}\right)=\infty$. \end{prop}
\begin{proof}
By Lemma \ref{lem:big g implies formula is constant}, for every $\varepsilon<\aleph_{1}$
there is some formula $\varphi_{\varepsilon}$ and a non-forking pattern
$\left\{ \varphi_{\varepsilon},\bar{a}_{\alpha}^{\varepsilon}\left|\,\alpha<\beth_{\varepsilon}\left(\kappa\right)\right.\right\} $
over a model $M_{\varepsilon}$ of size $\kappa$. We may assume that
$\varphi_{\varepsilon}=\varphi$ for all $\varepsilon<\aleph_{1}$.

Let $\psi$ be the following $L_{\omega_{1},\omega}$ sentence in
the language
\[
\left\{ P\left(x\right),S\left(x\right),Q\left(\alpha\right),<\left(\alpha,\beta\right),R\left(x,\alpha\right),<_{R}\left(x,y,\alpha\right)\right\} \cup L\left(T\right)
\]
 saying:
\begin{enumerate}
\item $S\models T$
\item $P$ is an $L$-elementary substructure of $S$.
\item $S\cap Q=\emptyset$
\item The universe is $S\cup Q$.
\item $Q$ is infinite and $<$ is a linear order on $Q$.
\item For each $\alpha\in Q$, $R\left(-,\alpha\right)$ is infinite and
contained in $S$ and $<_{R}\left(-,-,\alpha\right)$ is discrete
linear order with a first element on $R\left(-,\alpha\right)$.
\item For each $\alpha\in Q$, $R\left(-,\alpha\right)$ is an $L$-indiscernible
sequence over $P\cup\bigcup_{\beta<\alpha}R\left(-,\beta\right)$
ordered by $<_{R}\left(-,-,\alpha\right)$.
\item The set $\left\{ \varphi\left(x,y_{\alpha,0}\right)\land\neg\varphi\left(x,y_{\alpha,1}\right)\left|\,\alpha\in Q\right.\right\} $
does not fork over $P$ (in the sense of $L$), where $y_{\alpha,0}$
and $y_{\alpha,1}$ are the first elements in the sequence $R\left(-,\alpha\right)$.
\end{enumerate}
Note that (6) can be expressed in $L_{\omega_{1},\omega}$ by Lemma
\ref{lem:forking is L_omega1 definable}.

As the assumptions of Fact \ref{fac:L_omega1 sentence} are satisfied,
for each $\lambda$ we find a model $B$ of $\psi$ such that:
\begin{itemize}
\item $\left|P^{B}\right|=\aleph_{0}$
\item There is an embedding $h$ of $\left(\lambda,<\right)$ into $\left(Q^{B},<^{B}\right)$.
\end{itemize}
For all $\alpha<\lambda$ let $\bar{a}_{\alpha}$ be an infinite sub-sequence
of $R\left(B,h\left(\alpha\right)\right)$ and let $M=P\left(B\right)$.
By (1) -- (8), it follows that $\left\{ \varphi,\bar{a}_{\alpha}\left|\,\alpha<\lambda\right.\right\} $
is a non-forking pattern of depth $\lambda$ over $M$ --- as wanted. \end{proof}
\begin{cor}

\begin{enumerate}
\item If for all $\varepsilon<\aleph_{1}$, there is some $\kappa$ such
that $g_{T}\left(\kappa\right)>\beth_{\varepsilon}\left(\kappa\right)$
then $f_{T}\left(\lambda,\kappa\right)\geq\ded\left(\lambda\right)$
for all $\lambda\geq\kappa$.
\item If for every $\varepsilon<\aleph_{1}$ there is some $\lambda\geq\beth_{\varepsilon}\left(\kappa\right)$
such that $f_{T}\left(\lambda,\kappa\right)>\lambda^{<\beth_{\varepsilon}\left(\kappa\right)}$
then $f_{T}\left(\lambda,\kappa\right)\geq\ded\left(\lambda\right)$
for all $\lambda\geq\kappa$.
\item If $f_{T}\left(\lambda,\kappa\right)>\lambda^{<\beth_{\aleph_{1}}\left(\kappa\right)}$
for some $\lambda\geq\beth_{\aleph_{1}}\left(\kappa\right)$, then
$f_{T}\left(\lambda,\kappa\right)\geq\ded\left(\lambda\right)$ for
all $\lambda\geq\kappa$.
\end{enumerate}
\end{cor}
\begin{proof}
(1) By Lemma \ref{prop:If g is bigger than beth omega1 then it is infinite},
we know that $g_{T}\left(\aleph_{0}\right)=\infty$. For any $\lambda\geq\kappa$,
by Lemma \ref{lem:If g is bigger than theta then f is bigger than lambda^theta}
we have that $f_{T}\left(\kappa,\lambda\right)\geq\lambda^{\left\langle \theta\right\rangle _{\tr}}$
for all $\theta\leq\lambda$. As $\ded\left(\lambda\right)=\sup\left\{ \lambda^{\left\langle \theta\right\rangle _{\tr}}\left|\,\theta\leq\lambda,\,\mbox{is regular}\right.\right\} $
by Proposition \ref{prop:definitions of ded} (\ref{enu:ded is supremum of tree exponent})
we get $f_{T}\left(\kappa,\lambda\right)\geq\ded\left(\lambda\right)$.

(2) Let $\varepsilon<\aleph_{1}$ be a limit ordinal and $\theta=\beth_{\varepsilon}\left(\kappa\right)$.
Then
\[
\sum_{\mu<\theta}2^{2^{\kappa+\mu}}=\sum_{\alpha<\varepsilon}2^{2^{\beth_{\alpha}\left(\kappa\right)}}=\sum_{\alpha<\varepsilon}\beth_{\alpha+2}\left(\kappa\right)=\beth_{\varepsilon}\left(\kappa\right).
\]
By Lemma \ref{lem:polynomial f_T implies g_T bigger than degree},
$g_{T}\left(\kappa\right)>\beth_{\varepsilon}\left(\kappa\right)$.
So we can apply (1) to conclude.

(3) follows from (2).
\end{proof}

\section{\label{sec:Inside NTP2}Inside $\protect\NTPT$}

$\NTPT$ is a large class of first-order theories containing both
NIP and simple theories introduced by Shelah. For a general treatment,
see \cite{ChernikovNTP2}. In this section we show that for theories
in this class, the non-forking spectra is well behaved, i.e. it cannot
take values between (\ref{enu:ded kappa ^ omega}) and (\ref{enu:2 ^ lambda}).
\begin{fact}
(see e.g. \cite{HP}) Let $p\left(x\right)$ be a global type non-splitting
over a set $A$. For any set $B\supseteq A$, and an ordinal $\alpha$,
let the sequence $\bar{c}=\left\langle c_{i}\left|\,i<\alpha\right.\right\rangle $
be such that $c_{i}\models p|_{Bc_{<i}}$. Then $\bar{c}$ is indiscernible
over $B$ and its type over $B$ does not depend on the choice of
$\bar{c}$. Call this type $p^{\left(\alpha\right)}|_{B}$, and let
$p^{\left(\alpha\right)}=\bigcup_{B\supseteq A}p^{\left(\alpha\right)}|_{B}$.
Then $p^{\left(\alpha\right)}$ also does not split over $A$. \end{fact}
\begin{defn}
(strict invariance) Let $p\left(x\right)$ be a global type. We say
that $p$ is \emph{strictly invariant} over a set $A$ if $p$ does
not split over $A$, and if $B\supseteq A$ and $c\models p|_{B}$
then $\tp\left(B/cA\right)$ does not fork over $A$. \end{defn}
\begin{lem}
\label{lem:making p into an heir}Let $p$ be a global type finitely
satisfiable in $A$. Then there is some model $M\supseteq A$ with
$\left|M\right|\leq\left|A\right|+\aleph_{0}$ such that $p^{\left(\omega\right)}$
is strictly invariant over $M$. \end{lem}
\begin{proof}
Let $M_{0}$ be some model containing $A$ of size $\left|A\right|+\aleph_{0}$.
Construct by induction an increasing sequence of models $M_{i}$ for
$i<\omega$, such that $\left|M_{i}\right|=\left|M_{0}\right|$ and
for every formula $\varphi\left(x,y\right)$ over $M$ if $\varphi\left(x,c\right)\in p^{\left(\omega\right)}$
for some $c$, then there is some $c'\in M_{i+1}$ such that $\varphi\left(x,c'\right)\in p^{\left(\omega\right)}$.
Let $M=\bigcup_{i<\omega}M_{i}$.
\end{proof}
In lieu of giving a definition of $\NTPT$, we only state the properties
which we will be using.
\begin{fact}
\label{fac:Cheka}\cite{cheka} Let $T$ be $\NTPT$ and $M\models T$,
then:
\begin{enumerate}
\item $\varphi\left(x,c\right)$ divides over $M$ if and only if $\varphi\left(x,c\right)$
forks over $M$.
\item Let $p\left(x\right)$ is a global type strictly invariant over $M$
and $\left\langle c_{i}\left|\,i<\omega\right.\right\rangle \models p^{\left(\omega\right)}|_{M}$.
Then for any formula $\varphi\left(x,c_{0}\right)$ dividing over
$M$, $\left\{ \varphi\left(x,c_{i}\right)\left|\,i<\omega\right.\right\} $
is inconsistent.
\end{enumerate}
\end{fact}
Improving on \cite[Theorem 4.3]{cheka} we establish the following:
\begin{thm}
\label{thm:IP in NTP2}Let $T$ be $\NTPT$. Then the following are
equivalent:
\begin{enumerate}
\item $f_{T}\left(\kappa,\lambda\right)>\left(\ded\kappa\right)^{\aleph_{0}}$
for some $\lambda\geq\kappa$.
\item $T$ has IP.
\item $f_{T}\left(\kappa,\lambda\right)=2^{\lambda}$ for every $\lambda\geq\kappa$.
\end{enumerate}
\end{thm}
\begin{proof}
(1) implies (2) follows from Fact \ref{fac:NIP} and (3) implies (1)
is clear.

(2) implies (3): Fix $\lambda\geq\kappa$. Let $\varphi\left(x,y\right)$
have IP, and $\bar{a}=\left\langle a_{i}\left|\,i<\omega\right.\right\rangle $
be an indiscernible sequence such that $\forall U\subseteq\omega\exists b_{U}\,\varphi\left(a_{i},b_{U}\right)\Leftrightarrow i\in U$.
Let $p\left(x\right)$ be a global non-algebraic type finitely satisfiable
in $\bar{a}$. By Lemma \ref{lem:making p into an heir}, there a
model $M\supseteq\bar{a}$ be such that $\left|M\right|\leq\aleph_{0}$
and $p^{\left(\omega\right)}$ is strictly invariant over $M$.

Let $\bar{b}=\left\langle b_{i}\left|\,i<\lambda\right.\right\rangle $
realize $p^{\left(\lambda\right)}|_{M}$. We show that $p_{\eta}\left(x\right)=\left\{ \varphi\left(x,b_{i}\right)^{\mbox{if }\eta\left(i\right)=1}\left|\,i<\lambda\right.\right\} $
does not divide over $M$ for any $\eta\in2^{\lambda}$. 

First note
that $p_{\eta}\left(x\right)$ is consistent for any $\eta$, as $\tp\left(\bar{b}/M\right)$
is finitely satisfiable in $\bar{a}$. But as for any $k<\omega$,
$\left\langle \left(b_{k\cdot i},b_{k\cdot i+1},\ldots,b_{k\cdot\left(i+1\right)-1}\right)\left|\,i<\omega\right.\right\rangle $
realizes $\left(p^{\left(k\right)}\right)^{\left(\omega\right)}$,
Fact \ref{fac:Cheka}(2) implies that $p_{\eta}\left(x\right)|_{b_{0}\ldots b_{k-1}}$
does not divide over $M$ for any $k<\omega$. Thus by indiscernibility
of $\bar{b}$, $p_{\eta}(x)$ does not divide over $M$.

Take $N\supseteq\bar{b}\cup M$ of size $\lambda$. By Fact \ref{fac:Cheka}(1)
every $p_{\eta}$ extends to some $p'_{\eta}\in S^{\nf}\left(N,M\right)$,
thus $f_{T}\left(\kappa,\lambda\right)=2^{\lambda}$.
\end{proof}

\section{Examples}

\subsection{Examples of (\ref{enu:kappa}) -- (\ref{enu:ded kappa ^ omega}).}
\begin{prop}

\begin{enumerate}
\item If $T$ is the theory of equality, then $f_{T}\left(\kappa,\lambda\right)=\kappa$
for all $\lambda\geq\kappa$.
\item Let $T$ be the model companion of the theory of countably many unary
relations then $f_{T}\left(\kappa,\lambda\right)=\kappa+2^{\aleph_{0}}$
for all $\lambda\geq\kappa$.
\item Let $T$ be the model companion of the theory of countably many equivalence
relations then $f_{T}\left(\kappa,\lambda\right)=\kappa^{\aleph_{0}}$
for all $\lambda\geq\kappa$.
\item Let $T=DLO$. Then $f_{T}\left(\kappa,\lambda\right)=\ded\left(\kappa\right)$
for all $\lambda\geq\kappa$.
\item Let $T$ be the model companion of infinitely many linear orders.
Then $f_{T}\left(\kappa,\lambda\right)=\ded\left(\kappa\right)^{\aleph_{0}}$.
\end{enumerate}
\end{prop}
\begin{proof}
(1) -- (3): it is well known that these examples have the corresponding
$f_{T}\left(\kappa\right)$'s, and that they are stable. It follows
from Remark \ref{rem:stable theories} that they have the corresponding
$f_{T}\left(\kappa,\lambda\right)$.

(4): It is easy to check that every type has finitely many non-splitting
global extensions, but DLO is NIP so by Fact \ref{fac:non-splitting}
every non-forking extension is non-splitting. Since $f_{T}\left(\kappa\right)=\ded\left(\kappa\right)$
for this theory, we are done.

(5): This theory is NIP so $f_{T}\left(\kappa,\lambda\right)\leq\ded\left(\kappa\right)^{\aleph_{0}}$
by Fact \ref{fac:NIP}, and clearly $f_{T}\left(\kappa\right)=\left(\ded\kappa\right)^{\aleph_{0}}$.
\end{proof}

\subsection{Circularization.\protect \\
}

We shall first describe a general construction for examples of non-forking
spectra functions.

For this section, a ``formula'' means an $\emptyset$-definable
formula unless otherwise specified. Most formulas we work with are
partitioned formulas, $\varphi\left(\bar{x};\bar{y}\right)$, where
the variables are broken into two distinct sets. We write $\varphi$
instead of $\varphi\left(\bar{x};\bar{y}\right)$ when the partition
is clear from the context. We let $\varphi^{1}=\varphi$ and $\varphi^{0}=\neg\varphi$.
We assume that our languages relational in this section (so a subset
is a substructure).

\subsubsection{Circularization: Base step.\protect \\
}

The dense circular order was used as an example of a theory where
forking is not the same as dividing (see e.g. \cite[Example 2.11]{KimThesis}).
The reason is that with circular ordering around, it is hard not to
fork.
\begin{defn}
A \emph{circular order} on a finite set is a ternary relation obtained
by placing the points on a circle and taking all triples in clockwise
order. For an infinite set, a circular order is a ternary relation
such that the restriction to any finite set is a circular order. Equivalently,
a circular order is a ternary relation $C$ such that for every $x$,
$C\left(x,-,-\right)$ is a linear order on $\left\{ y\left|\,y\neq x\right.\right\} $
and $C\left(x,y,z\right)\to C\left(y,z,x\right)$ for all $x,y,z$.
Denote the theory of circular orders by $T_{C}$.
\end{defn}
The following definitions are well-known.
\begin{defn}
Let $K$ be a class of $L$-structures (where $L$ is relational).
\begin{enumerate}
\item We say that $K$ has the \emph{strong amalgamation property} \emph{(SAP)}
if for every $A,B,C\in K$ and embeddings $i_{1}:A\to B$ and $i_{2}:A\to C$
there exist both a structure $D\in K$ and embeddings $j_{1}:B\to D$,
$j_{2}:C\to D$ such that

\begin{enumerate}
\item $j_{1}\circ i_{1}=j_{2}\circ i_{2}$ and
\item $j_{1}\left(B\right)\cap j_{2}\left(C\right)=(j_{1}\circ i_{1})\left(A\right)=(j_{2}\circ i_{2})\left(A\right)$.
\end{enumerate}
\item We say that $K$ has the \emph{disjoint embedding property} \emph{(DEP)}
if for any 2 structures $A,B\in K$, there exists a structure $C\in K$
and embeddings $j_{1}:B\to C$, $j_{2}:A\to C$ such that $j_{1}\left(A\right)\cap j_{2}\left(B\right)=\emptyset$.
\item We say that a first-order theory $T$ has these properties if its
class of (finite) models has them.
\end{enumerate}
\end{defn}
Note that
\begin{rem}
$T_{C}$ is universal and it has DEP and SAP.\end{rem}
\begin{fact}
\label{fac:Hodges}Let $T$ be a universal theory with DEP and SAP
in a finite relational language $L$, then:
\begin{enumerate}
\item \cite[Theorem 7.4.1]{Hod} It has a model completion $T_{0}$ which
is $\omega$-categorical and eliminates quantifiers.
\item \cite[Theorem 7.1.8]{Hod} If $A\subseteq M\models T_{0}$ then $\acl\left(A\right)=A$.
\end{enumerate}
\end{fact}
\begin{cor}
\label{cor:infinite-1}Suppose that $\varphi\left(\bar{x};\bar{y}\right)$
is a formula in $L$, $\bar{a}\in M\models T_{0}$. If $M\models\exists\bar{z}\varphi\left(\bar{z};\bar{a}\right)\land\bar{z}\nsubseteq\bar{a}$
then $\left\{ \bar{t}\in M\left|\,\varphi\left(\bar{t};\bar{a}\right)\right.\right\} $
is infinite.\end{cor}
\begin{defn}
For any formula $\varphi\left(\bar{x};\bar{y}\right)$ in $L$ where
$\bar{x}$ is not empty , let $C\left[\varphi\left(\bar{x};\bar{y}\right)\right]$
be a new $\lg\left(\bar{y}\right)+3\cdot\lg\left(\bar{x}\right)$-place
relation symbol. Denote $L\left[\varphi\left(\bar{x};\bar{y}\right)\right]=L\cup\left\{ C\left[\varphi\left(\bar{x};\bar{y}\right)\right]\right\} $.
\end{defn}

\begin{defn}
Suppose $\varphi\left(\bar{x};\bar{y}\right)$ is a quantifier free
formula in $L$ with $\bar{x}$ not empty. Let $T\left[\varphi\left(\bar{x};\bar{y}\right)\right]$
be the theory in $L\left[\varphi\left(\bar{x};\bar{y}\right)\right]$
containing $T$ and the following axioms:
\begin{itemize}
\item For all $\bar{t}$ in the length of $\bar{y}$, the set:
\[
S\left[\varphi\left(\bar{x};\bar{y}\right)\right]\left(\bar{t}\right):=\left\{ \bar{s}\left|\,\bar{s}\cap\bar{t}=\emptyset\land\lg\left(\bar{s}\right)=\lg\left(\bar{x}\right)\land\varphi\left(\bar{s};\bar{t}\right)\right.\right\}
\]
is circularly ordered by the relation:
\[
C\left[\varphi\left(\bar{x};\bar{y}\right)\right]\left(\bar{t}\right):=\left\{ \left(\bar{s}_{1},\bar{s}_{2},\bar{s}_{3}\right)\left|\,C\left[\varphi\left(\bar{x},\bar{y}\right)\right]\left(\bar{t},\bar{s}_{1},\bar{s}_{2},\bar{s}_{3}\right)\right.\right\}
\]
 (i.e. $C\left[\varphi\left(\bar{x};\bar{y}\right)\right]$ with index
$\bar{t}$ orders this set in a circular order). Call $\bar{t}$ the
index variables, and $\bar{s}$ the main variables.
\item If $C\left[\varphi\left(\bar{x};\bar{y}\right)\right]\left(\bar{t}\right)\left(\bar{s}_{1},\bar{s}_{2},\bar{s}_{3}\right)$
then $\bar{s}_{1},\bar{s}_{2},\bar{s}_{3}\in S\left[\varphi\left(\bar{x};\bar{y}\right)\right]\left(\bar{t}\right)$.
\end{itemize}
\end{defn}
\begin{claim}
\label{cla:Thas}If $\varphi$ is as in the definition, then
\begin{enumerate}
\item $T\left[\varphi\right]$ is universal.
\item $T\left[\varphi\right]$ has DEP.
\item $T\left[\varphi\right]$ has SAP.
\end{enumerate}
\end{claim}
\begin{proof}
As $T_{C}$ is universal, (1) is clear (note that this uses the fact
that $\varphi$ is quantifier free).

(3): Let $M_{0}'$, $M_{1}'$ and $M_{2}'$ be models of $T\left[\varphi\right]$
such that $M_{0}'=M_{1}'\cap M_{2}'$. Let $M_{i}=M_{i}'\upharpoonright L$
for $i<3$. By assumption, there is a model $M_{3}\models T$ such
that $M_{1}\cup M_{2}\subseteq M_{3}$. We define $M_{3}'$ as an
expansion of $M_{3}$. Let $\bar{t}\in M_{3}$ be a tuple of length
$\lg\left(\bar{y}\right)$. Split into cases:
\begin{casenv}
\item \label{cas:M_0}$\bar{t}\in M_{0}'$. In this case, $\left(S^{M'_{i}}\left[\varphi\right]\left(\bar{t}\right),C^{M'_{i}}\left[\varphi\right]\left(\bar{t}\right)\right)$
are circular orders for $i<3$ and $S^{M'_{1}}\left[\varphi\right]\left(\bar{t}\right)\cap S^{M'_{2}}\left[\varphi\right]\left(\bar{t}\right)=S^{M'_{0}}\left[\varphi\right]\left(\bar{t}\right)$
so we can amalgamate them as circular orders and extend it arbitrarily
to $S^{M_{3}}\left[\varphi\right]\left(\bar{t}\right)$, and that
will be $C^{M_{3}'}\left[\varphi\right]\left(\bar{t}\right)$.

Note that in the special case where $S^{M_{0}}\left[\varphi\right]\left(\bar{t}\right)=\emptyset$,
there are no restrictions on the place of $S^{M_{i}}\left[\varphi\right]\left(\bar{t}\right)$
for $i<3$ in this order.

\item $\bar{t}\in M_{1}\backslash M_{2}$. Then $\left(S^{M_{1}'}\left[\varphi\right]\left(\bar{t}\right),C^{M_{1}'}\left[\varphi\right]\left(\bar{t}\right)\right)$
is a circular order. Extend it so that its domain would be $S^{M_{3}}\left[\varphi\right]\left(\bar{t}\right)$
arbitrarily.
\item $\bar{t}\in M_{2}\backslash M_{1}$ --- the same.
\item $\bar{t}\notin M_{1}$ and $\bar{t}\notin M_{2}$. Then $C^{M_{3}'}\left[\varphi\right]\left(\bar{t}\right)$
is any circular order on $S^{M_{3}}\left[\varphi\right]\left(\bar{t}\right)$.
\end{casenv}
(2): Similar to (3), but easier. \end{proof}
\begin{rem}
\label{rem:AmalInside} It is follows from the proof of amalgamation,
that if $M\models T$ contains models $M_{0}\subseteq M_{i}\subseteq M$
for $i<n$ such that $M_{0}=M_{i}\cap M_{j}$ for $i<j<n$ and for
each $M_{i}$, there is an expansion $M_{i}'$ to a model of $T\left[\varphi\right]$
such that $M_{0}'\subseteq M_{i}'$ then there is an expansion $M'$
of $M$ to a model of $T\left[\varphi\right]$ such that $M_{i}'\subseteq M'$. \end{rem}
\begin{claim}
\label{cla:existence}$ $
\begin{enumerate}
\item If $M\models T$, then we can expand it to a model $M'$ of $T\left[\varphi\right]$.
\item Moreover: if $B\subseteq M$ and there is already an expansion $B'$
of $B$ to a model of $T\left[\varphi\right]$, then we can expand
$M$ in such a way that $B'\subseteq M'$.
\item Moreover: suppose that

\begin{itemize}
\item $A\subseteq M$
\item $\left\langle \bar{c}_{i}\left|\,i<n\right.\right\rangle $ is a finite
sequence of finite tuples from $M$, such that $\bar{c}_{i}\cap\bar{c}_{j}\subseteq A$,
$\tp_{\qf}\left(\bar{c}_{i}/A\right)=\tp_{\qf}\left(\bar{c}_{j}/A\right)$
for all $i<j<n$.
\item $M_{0}'$ is an expansion of $A\bar{c}_{0}$ to a model of $T\left[\varphi\right]$.
\end{itemize}

Then we can find an expansion $M'$ such that the quantifier free
types are still equal in the sense of $L\left[\varphi\right]$ and
$M_{0}'\subseteq M'$.

\end{enumerate}
\end{claim}
\begin{proof}
(2): For any $\bar{t}$ in the length of $\bar{y}$, if $\bar{t}\in B$
then we choose a circular order $C^{M'}\left[\varphi\right]\left(\bar{t}\right)$
that extends $C^{B'}\left[\varphi\right]\left(\bar{t}\right)$ on
$S^{M}\left[\varphi\right]\left(\bar{t}\right)$. If not, then define
it arbitrarily.

(3): Let $M_{i}=A\bar{c}_{i}$. As $\bar{c}_{0}\equiv_{A}^{\qf}\bar{c_{i}}$
for $i<n$, there are isomorphisms $f_{i}:M_{0}\to M_{i}$ of $L$
that fix $A$ and take $\bar{c}_{0}$ to $\bar{c}_{i}$. So $f_{i}$
induces expansions $M_{i}'$ of $M_{i}$, isomorphic (via $f_{i}$)
to $M_{0}'$. As the intersection of any two models $M_{i}$ is exactly
$A$, by Remark \ref{rem:AmalInside}, there is an expansion $M'$
of $M$ to a model of $T\left[\varphi\right]$ that contains $M_{i}'$.
In this expansion the quantifier free types will remain the same because
$f_{i}$ are $L\left[\varphi\right]$-isomorphisms.\end{proof}
\begin{cor}
\label{cor:exClosPer}Suppose that $M'\models T\left[\varphi\right]$,
$M'\upharpoonright L\subseteq N\models T$. Then there is an expansion
of $N$ to a model $N'$ of $T\left[\varphi\right]$ such that $M'\subseteq N'$.
In particular, if $M'\models T\left[\varphi\right]$ is existentially
closed, then $M'\upharpoonright L$ is an existentially closed model
of $T$. Denote by $T_{0}\left[\varphi\right]$ the model completion
of $T\left[\varphi\right]$. We will call it the \underline{$\varphi$-circularization}
of $T_{0}$. It follows that $T_{0}\left[\varphi\right]\upharpoonright L=T_{0}$
(for more see \cite[Theorem 8.2.4]{Hod}).
\end{cor}
We turn to dividing:
\begin{claim}
\label{cla:CircDividing} Assume that $M\models T_{0}\left[\varphi\right]$,
$A\subseteq M$, $\bar{a}\in M$, $S^{M}\left[\varphi\right]\left(\bar{a}\right)\cap A^{\lg\left(\bar{x}\right)}=\emptyset$,
and $\bar{c}\neq\bar{d}\in S^{M}\left[\varphi\right]\left(\bar{a}\right)$.
Then the formula $\psi\left(\bar{z};\bar{a},\bar{c},\bar{d}\right)=C\left[\varphi\right]\left(\bar{a},\bar{c},\bar{z},\bar{d}\right)$
2-divides over $A\bar{a}$. \end{claim}
\begin{proof}
Let $M_{0}=A\bar{a}$, $M_{1}=M_{0}\bar{c}\bar{d}$ and $M_{2}=M_{0}\bar{c}'\bar{d}'$
where $M_{1}\cap M_{2}=M_{0}$ and there is an isomorphism $f:M_{1}\to M_{2}$
that fixes $M_{0}$ and takes $\bar{c}\bar{d}$ to $\bar{c}'\bar{d}'$.

By SAP, there is a model $M_{3}\models T\left[\varphi\right]$ that
contains $M_{1}\cup M_{2}$. We wish to choose it carefully: in the
proof of Claim \ref{cla:Thas}, we saw that there are no constraints
on the amalgamation of $C^{M_{1}}\left[\varphi\right]\left(\bar{a}\right)$
and $C^{M_{2}}\left[\varphi\right]\left(\bar{a}\right)$ (because
$S^{M_{0}}\left[\varphi\right]\left(\bar{a}\right)=\emptyset$, see
the definition of $S\left[\varphi\right]$). In particular we can
put $\bar{c}'$ and $\bar{d}'$ so that in the circular order we have
$\bar{c}\to\bar{d}\to\bar{c}'\to\bar{d}'\to\bar{c}$, and in this
case there is no $\bar{z}$ such that $C\left[\varphi\right]\left(\bar{a}\right)\left(\bar{c},\bar{z},\bar{d}\right)$
and $C\left[\varphi\right]\left(\bar{a}\right)\left(\bar{c}',\bar{z},\bar{d}'\right)$.

Applying the same technique $n$ times, there is a model of $T\left[\varphi\right]$
with a sequence $\left\langle \bar{c}_{i},\bar{d}_{i}\left|i<n\right.\right\rangle $
that contains $M_{1}$ and satisfies $\tp_{\qf}\left(\bar{c}_{i}\bar{d}_{i}/A\bar{a}\right)=\tp_{\qf}\left(\bar{c}\bar{d}/A\bar{a}\right)$,
so that in the circular order $C\left[\varphi\right]\left(\bar{a}\right)$
the tuples will be ordered as follows: $\bar{c}\to\bar{d}\to\bar{c}_{1}\to\bar{d}_{1}\to\ldots\to\bar{c}_{n}\to\bar{d}_{n}\to\bar{c}$.
Hence, there is a model of $T_{0}\left[\varphi\right]$ and an infinite
such sequence, and this sequence witnesses the 2-dividing of $\psi\left(\bar{z};a,\bar{c},\bar{d}\right)$.

Note that the tuples $\bar{c}_{i}\bar{d}_{i}$ were chosen so that
the intersection of each pair $\bar{c}_{i}\bar{d}_{i}$, $\bar{c}_{j}\bar{d}_{j}$
is contained in $A$.
\end{proof}
The last sentence justifies the following auxiliary definition which
will make life a bit easier:
\begin{defn}
Say that a formula $\varphi\left(\bar{x},\bar{a}\right)$ $k$-\emph{divides
disjointly} over $A$ if there is an indiscernible sequence $\left\langle \bar{a}_{i}\left|\,i<\omega\right.\right\rangle $
that witnesses $k$-dividing and moreover $\bar{a}_{i}\cap\bar{a}_{j}\subseteq A$.\end{defn}
\begin{rem}
Note that if $\varphi\left(\bar{x},\bar{a}\right)$ divides over $A$,
then it divide disjointly over some $B\supseteq A$ (if $I$ is an
indiscernible sequence witnessing dividing, then $B=A\cup\bigcap I$).
\end{rem}
We shall also need some kind of a converse to the last claim. More
precisely, we need to say when a formula does not divide.
\begin{claim}
\label{cla:ConverseToDividing}Suppose
\begin{enumerate}
\item $A\subseteq M\models T_{0}\left[\varphi\right]$
\item $p\left(\bar{x}\right)=p_{1}\left(\bar{x}\right)\cup p_{2}\left(\bar{x}\right)$
is a complete quantifier-free type over $M$.
\item $p_{1}\left(\bar{x}\right)$ is a complete $L$ type over $M$ and
$p_{2}\left(\bar{x}\right)$ is a complete $\left\{ C\left[\varphi\right]\right\} $
type over $M$.
\item $p_{1}\left(\bar{x}\right)$ does not divide over $A$ (as an $L$-type
so also as an $L\left[\varphi\right]$-type).
\item For all $\bar{t}\in M^{\lg\left(\bar{y}\right)}$, $p_{2}\left(\bar{x}\right)\upharpoonright\left\{ C\left[\varphi\right]\left(\bar{t},-,-,-\right)\right\} $
does not divide over $A\bar{t}$ (this means all formulas in $p_{2}\left(\bar{x}\right)$
of the form $C\left[\varphi\right]\left(\bar{t},\bar{z}_{1},\bar{z}_{2},\bar{z}_{3}\right)$
where $\bar{x}$ substitutes the $\bar{z}$'s in some places and in
the others there are parameters from $M$).
\end{enumerate}
Then $p\left(\bar{x}\right)$ does not divide over $A$.

In particular, if both $p_{1}\left(\bar{x}\right)$, $p_{2}\left(\bar{x}\right)$
do not divide over $A$, then $p\left(\bar{x}\right)$ does not divide
over $A$.\end{claim}
\begin{proof}
Denote $\bar{x}=\left(x_{0},\ldots,x_{m-1}\right)$, $p\left(\bar{x},M\right)=p\left(\bar{x}\right)$.
We may assume that $p\upharpoonright x_{i}$ is non-algebraic for
all $i<m$ (otherwise, by Fact \ref{fac:Hodges}, $\left(x_{i}=c\right)\in p$
for some $c\in M$, so $c\in A$ as $x=c$ divides over $A$, and
we can replace $x_{i}$ by $c$). Suppose $\left\langle M_{i}\left|\,i<\omega\right.\right\rangle $
is an $L\left[\varphi\right]$-indiscernible sequence over $A$ in
some model $N\supseteq M$ such that $M_{0}=M$. We will show that
$\bigcup\left\{ p\left(\bar{x},M_{i}\right)\left|i<\omega\right.\right\} $
is consistent.

Let $\bar{c}\models\bigcup\left\{ p_{1}\left(\bar{x},M_{i}\right)\right\} $
(exists by (4)), and $B=\bigcup\left\{ M_{i}\left|i<\omega\right.\right\} $
and let $B'=B\bar{c}\upharpoonright L$ (i.e. forget $C\left[\varphi\right]$).
Also let $\bar{d}\models p\left(\bar{x}\right)$ be in some other
model $N'=M\bar{d}$ of $T\left[\varphi\right]$.

For $\bar{t}\in\left(B\bar{c}\right)^{\lg\left(\bar{y}\right)}$ we
define a circular order on $S\left[\varphi\right]\left(\bar{t}\right)$
to make $B'$ into a model $U$ of $T\left[\varphi\right]$ extending
$B$ such that $\bar{c}\models\bigcup\left\{ p\left(\bar{x},M_{i}\right)\right\} $.
\begin{casenv}
\item $\bar{t}\nsubseteq M_{i}\bar{c}$ for any $i<\omega$. In this case,
there is no information on $C\left[\varphi\right]\left(\bar{t}\right)$
in $\bigcup\left\{ p_{2}\left(\bar{x},M_{i}\right)\right\} $, so
let $C\left[\varphi\right]^{U}\left(\bar{t}\right)$ be any circular
order on $S\left[\varphi\right]\left(\bar{t}\right)$ that extends
the circular order $C\left[\varphi\right]^{B}\left(\bar{t}\right)$
(in case $\bar{t}\subseteq B$).
\item $\bar{t}\subseteq M_{i}\bar{c}$ for some $i<\omega$, but $\bar{t}\nsubseteq M_{j}\bar{c}$
for some other $j\neq i$. By indiscernibility, it follows that $\bar{t}\not\subseteq M_{j}\bar{c}$
for all $j\neq i$. Let $\sigma:M_{i}\bar{c}\to M\bar{d}$ be an $L$-isomorphism.
There are two sub-cases:

\begin{casenv}
\item $\bar{t}\cap\bar{c}\neq\emptyset$. Let $C\left[\varphi\right]^{U}\left(\bar{t}\right)$
 be any extension of $\sigma^{-1}\left(C\left[\varphi\right]^{N'}\left(\sigma\left(\bar{t}\right)\right)\right)$
to $S^{U}\left[\varphi\right]\left(\bar{t}\right)$.
\item $\bar{t}\cap\bar{c}=\emptyset$. Then $C\left[\varphi\right]^{B}\left(\bar{t}\right)$
is already a circular order on $S^{B}\left[\varphi\right]\left(\bar{t}\right)$.
On the other hand, $\sigma^{-1}\left(C\left[\varphi\right]^{N'}\left(\sigma\left(\bar{t}\right)\right)\right)$
defines some circular order on $S^{M_{i}\bar{c}}\left[\varphi\right]\left(\bar{t}\right)$.
The intersection is $S^{M_{i}}\left[\varphi\right]\left(\bar{t}\right)$
on which they agree, so we can amalgamate the two circular orders.
\end{casenv}
\item $\bar{t}\subseteq\bigcap M_{i}.$ In this case, by (5), $p_{2}\left(\bar{x}\right)\upharpoonright\left\{ C\left[\varphi\right]\left(\bar{t},-,-,-\right)\right\} $
does not divide over $A\bar{t}$, so let $\bar{c}'\models\bigcup\left\{ p_{2}\left(\bar{x},M_{i}\right)\upharpoonright C\left[\varphi\right]\left(\bar{t},-,-,-\right)\left|i<\omega\right.\right\} $.
Let $U'$ be the $L\left[\varphi\right]$ structure $B\bar{c}'.$
Let $f:B\bar{c}\to B\bar{c}'$ fix $B$ and take $\bar{c}$ to $\bar{c}'$.
Now, $C^{U'}\left[\varphi\right]\left(f\left(\bar{t}\right)\right)$
induces a circular order on
\[
S=f^{-1}\left(S^{U'}\left[\varphi\right]\left(f\left(\bar{t}\right)\right)\right)\cap S^{B'}\left[\varphi\right]\left(\bar{t}\right).
\]
Extend it to some circular order on $S^{U}\left[\varphi\right]\left(\bar{t}\right)$
and let it be $C^{U}\left[\varphi\right]\left(\bar{t}\right)$.
\item $\bar{t}\subseteq\bigcap M_{i}\bar{c}$, and $\bar{t}\cap\bar{c}\neq\emptyset$.
Let $\sigma_{i}:M_{i}\bar{c}\to M\bar{d}$ be the $L$-isomorphism
fixing $\bigcap M_{i}$ and taking $\bar{c}$ to $\bar{d}$. $\sigma_{i}$
induces a circular order on $S^{M_{i}\bar{c}}\left[\varphi\right]\left(\bar{t}\right)$,
and the intersection of any two $S^{M_{i}\bar{c}}\left[\varphi\right]\left(\bar{t}\right)$
and $S^{M_{j}\bar{c}}\left[\varphi\right]\left(\bar{t}\right)$ is
$S^{\bigcap M_{i}\bar{c}}\left[\varphi\right]\left(\bar{t}\right)$
on which these circular orders agree. By amalgamation, we have a circular
order on the union $\bigcup_{i}S^{M_{i}\bar{c}}\left[\varphi\right]\left(\bar{t}\right)$
that we can expand to a circular order on $S^{U}\left[\varphi\right]\left(\bar{t}\right)$.
\end{casenv}
\end{proof}
\begin{claim}
\label{cla:dividingPer}Let $A\subseteq M\models T_{0}\left[\varphi\right]$
be $\left|A\right|^{+}$-saturated and $M'=M\upharpoonright L$. Suppose
that $\psi\left(\bar{z},\bar{a}\right)$, a quantifier free $L$-formula,
$k$-divides disjointly over $A$ in $M'$. Then the same is true
in $M$.\end{claim}
\begin{proof}
Suppose that $I=\left\langle \bar{a}_{i}\left|\,i<\omega\right.\right\rangle \subseteq M$
witnesses $k$-dividing disjointly of $\psi\left(\bar{z},\bar{a}\right)$
over $A$ in the sense of $L$. Assume that $\bar{a}_{0}=\bar{a}$.

By Claim \ref{cla:existence} (3) and compactness, we can expand and
extend $M'$ to $M''\models T_{0}\left[\varphi\right]$ that will
keep the equality of types of the tuples in the sequence. In addition,
the interpretation of the new relation $C\left[\varphi\right]$ on
$A\bar{a}$ remains as it was in $M$. In particular, in $M''$, $\psi\left(\bar{z},\bar{a}\right)$
still $k$-divides over $A$. We may amalgamate a copy of $M''$ with
$M$ over $A\bar{a}$ to get a bigger model in which $\psi\left(\bar{z},\bar{a}\right)$
still $k$-divides disjointly and by saturation this is still true
in $M$.
\end{proof}

\subsubsection{\label{sub:General-construction}Circularization: Iterations.\protect \\
}

Assume there are theories $\mathcal{T}=\left\langle T_{i}^{\forall}\left|\,i\leq\omega\right.\right\rangle $
and formulas $\left\langle \varphi_{i}\left(\bar{x}_{i};\bar{y}_{i}\right)\left|\,i<\omega\right.\right\rangle $
in the finite relational languages $\left\langle L_{i}\left|\,i\leq\omega\right.\right\rangle $
where:
\begin{itemize}
\item $T_{0}^{\forall}$ is a universal theory with SAP and DEP in $L_{0}$.
\item $T_{i}^{\forall}$ is a theory in $L_{i}$ for $i\leq\omega$.
\item $\varphi_{i}\left(\bar{x}_{i};\bar{y}_{i}\right)$ is a quantifier
free formula in $L_{i}$.
\item $L_{i}=L_{i}\left[\varphi_{i}\left(\bar{x}_{i};\bar{y}_{i}\right)\right]$
and $T_{i+1}^{\forall}=T_{i}^{\forall}\left[\varphi_{i}\left(\bar{x}_{i};\bar{y}_{i}\right)\right]$.
\item $L_{\omega}=\bigcup\left\{ L_{i}\left|i<\omega\right.\right\} $ and
$T_{\omega}^{\forall}=\bigcup\left\{ T_{i}^{\forall}\left|i<\omega\right.\right\} $. \end{itemize}
\begin{prop}
\label{prop:model completion of phi-circulization}In the situation
above, $T_{i}^{\forall}$ has a model completion $T_{i}$, $T_{i}\subseteq T_{i+1}$
and $T_{i}\subseteq T_{\omega}$ which is the model completion of
$T_{\omega}^{\forall}$ for all $i<\omega$.\end{prop}
\begin{proof}
Follows from Claim \ref{cla:Thas} and Claim \ref{cor:exClosPer}.
\end{proof}
From now on, we work in $T:=T_{\omega}$. Call $T_{\omega}$ the $\bar{\varphi}$-circularization
of $T_{0}$ where $\bar{\varphi}=\left\langle \varphi_{i}\left|\,i<\omega\right.\right\rangle $.
Let $M\models T$ and $A\subseteq M$.
\begin{claim}
\label{cla:ForkingIffFS}Suppose $\varphi\left(\bar{x};\bar{y}\right)=\varphi_{i}\left(\bar{x}_{i};\bar{y}_{i}\right)$
for some $i<\omega$. Then for all $\bar{a}\in M^{\lg\left(\bar{y}\right)}$,
$\varphi\left(\bar{z},\bar{a}\right)\land\left(\bar{z}\cap\left(\bar{a}\cap A\right)=\emptyset\right)$
forks over $A$ if and only if it is not satisfied in $A$. \end{claim}
\begin{proof}
Denote $\bar{a}'=\bar{a}\cap A,$ and $\alpha\left(\bar{z},\bar{a}\right)=\varphi\left(\bar{z},\bar{a}\right)\wedge\left(\bar{z}\cap\bar{a}'=\emptyset\right)$.
Obviously if $\alpha$ is satisfied in $A$ it does not fork over
$A$.

Suppose $\alpha$ is not satisfied in $A$. Consider the formula $\psi\left(\bar{z},\bar{a}\right)=\varphi\left(\bar{z},\bar{a}\right)\wedge\left(\bar{z}\cap\bar{a}=\emptyset\right)$.
First we prove that $\psi$ forks. It defines $S\left[\varphi\right]^{M}\left(\bar{a}\right)$,
and by assumption $S\left[\varphi\right]^{M}\left(\bar{a}\right)\cap A=\emptyset$.
Note that for all $\bar{c}\neq\bar{d}\in S^{M}\left[\varphi\right]\left(\bar{a}\right)$,
since $C^{M}\left[\varphi\right]\left(\bar{a}\right)$ orders this
set in a circular order,
\[
S\left[\varphi\right]\left(\bar{a}\right)\left(\bar{z}\right)\vdash C\left[\varphi\right]\left(\bar{a}\right)\left(\bar{c},\bar{z},\bar{d}\right)\vee C\left[\varphi\right]\left(\bar{a}\right)\left(\bar{d},\bar{z},\bar{c}\right)\vee\bar{z}=\bar{c}\vee\bar{z}=\bar{d}.
\]
 If $S\left[\varphi\right]^{M}\left(\bar{a}\right)=\emptyset$ we
are done. If not, (by Corollary \ref{cor:infinite-1}) this set is
infinite and there are such $\bar{c},\bar{d}$.

By Claim \ref{cla:CircDividing} and Claim \ref{cla:dividingPer},
it follows that $C\left[\varphi\right]\left(\bar{a}\right)\left(\bar{c},\bar{z},\bar{d}\right)$,
$C\left[\varphi\right]\left(\bar{a}\right)\left(\bar{d},\bar{z},\bar{c}\right)$
divides over $A\bar{a}$. By Corollary \ref{cor:infinite-1}, both
$\bar{z}=\bar{c}$ and $\bar{z}=\bar{d}$ divides over $A\bar{a}$.
This means that $S\left[\varphi\right]\left(\bar{a}\right)\left(\bar{z}\right)=\psi\left(\bar{z},\bar{a}\right)$
forks over $A$.

Now, $\alpha\left(\bar{z},\bar{a}\right)\vdash\psi\left(\bar{z},\bar{a}\right)\vee\bigvee_{i,j}\left(z_{i}=a_{j}\right)$
(where $z_{i},a_{j}$ run over all the variables and parameters from
$\bar{a}\backslash A$ in $\varphi$). But the formula $z_{i}=a_{j}$
divides over $A$ when $a_{j}\notin A$ (By Corollary \ref{cor:infinite-1}),
so we are done.
\end{proof}
On the other hand, we have:
\begin{claim}
\label{cla:NotDiv}Suppose that $p\left(\bar{x}\right)$ is a (quantifier
free) type over $M$ such that:
\begin{itemize}
\item $p_{0}\left(\bar{x}\right)=p\upharpoonright L_{0}$ does not divide
over $A$.
\item $p_{i}\left(\bar{x}\right)=p\upharpoonright L_{i+1}\backslash L_{i}$
does not divide over $A$.
\end{itemize}
Then $p$ does not divide over $A$.\end{claim}
\begin{proof}
By induction on $i<\omega$ we show that $p'_{i}=p\upharpoonright L_{i}$
does not divide over $A$. For $i=0$ it is given. For $i+1$ use
Claim \ref{cla:ConverseToDividing}.
\end{proof}
The following definition is a bit vague
\begin{prop}
\label{def:Circul iteration} Let $\F$ be a function defined on the
class of all countable relational first-order languages such that
$\mathcal{\F}\left(L\right)$ is a set of quantifier free partitioned
formulas in $L$. Let $T_{0}$ be a universal theory in the language
$L_{0}$ satisfying SAP and DEP. We define:
\begin{itemize}
\item For $n<\omega$, let $L_{n+1}=\bigcup\left\{ L_{n}\left[\varphi\left(\bar{x};\bar{y}\right)\right]\left|\,\varphi\left(\bar{x};\bar{y}\right)\in\mathcal{F}\left(L_{n}\right)\right.\right\} $
and let $L_{\omega}$ be their union $\bigcup\left\{ L_{n}\left|n<\omega\right.\right\} .$
\item For $n<\omega$, let $T_{n}^{\forall}$ be a universal theory in $L_{n}$
defined by induction on $n\leq\omega$:

\begin{itemize}
\item $T_{0}^{\forall}=T_{0}$
\item $T_{n+1}^{\forall}=\bigcup\left\{ T_{n}^{\forall}\left[\varphi\left(\bar{x};\bar{y}\right)\right]\left|\,\varphi\in\mathcal{F}\left(L_{n}\right)\right.\right\} $
\item $T_{\omega}^{\forall}=\bigcup\left\{ T_{n}^{\forall}\left|\,n<\omega\right.\right\} $
\end{itemize}
\end{itemize}
Then $T_{\omega}^{\forall}$ has a model completion which we denote
by $\circlearrowright_{T_{0},L_{0},\F}$. Moreover, it is a $\bar{\varphi}$-circularization
for some choice of $\bar{\varphi}$.\end{prop}
\begin{proof}
By carefully choosing an enumeration of the formulas in $L_{\omega}$,
we can reconstruct $T_{\omega}^{\forall},L_{\omega}$ in such a way
that in each step we deal with one formula and it has a model completion
by Proposition \ref{prop:model completion of phi-circulization}.
\end{proof}

\subsection{Example of (\ref{enu:double exponent kappa}).}
\begin{defn}
Let $L_{0}=\left\{ =\right\} $ and $T_{0}$ be empty. Let $\F\left(L\right)$
be the set of all quantifier free partitioned formulas from $L$.
Let $T=\circlearrowright_{T_{0},L_{0},\F}$. \end{defn}
\begin{rem}
\label{rem:IP}$T$ has IP: Let $\varphi\left(x,y\right)=\left(x\neq y\right)$.
Then $C\left[\varphi\right]\left(y;x_{1},x_{2},x_{3}\right)$ has
IP. \end{rem}
\begin{cor}
For any set $A$, a type $p\left(\bar{x}\right)\in S\left(\C\right)$
does not fork over $A$ if and only if $p$ is finitely satisfiable
in $A$. In particular, by Fact \ref{fac:non-splitting}, $f_{T}\left(\kappa,\lambda\right)\leq2^{2^{\kappa}}$. \end{cor}
\begin{proof}
Suppose $p\left(\bar{x}\right)$ is a global type that is not finitely
satisfiable in $A$. By quantifier elimination, there is a quantifier
free formula $\varphi\left(\bar{x};\bar{y}\right)$ and $\bar{a}\in\C$
such that $\varphi\left(\bar{x},\bar{a}\right)\in p$ and this formula
is not satisfiable in $A$. If $\bar{a}\cap A\neq\emptyset$, and
$x_{i}=a\in p$ for some $a\in\bar{a}\cap A$, replace $x_{i}$ by
$a$ in $\varphi$, and change the partition of the variables so that
we get $\varphi\left(\bar{z},\bar{a}\right)\land\bar{z}\cap\left(\bar{a}\cap A\right)=\emptyset\in p$.
By Claim \ref{cla:ForkingIffFS}, this formula forks over $A$ and
we are done. \end{proof}
\begin{prop}
\label{prop:f_T =00003D 2^2^kappa}We have $f_{T}\left(\kappa,\lambda\right)=2^{\min\left\{ 2^{\kappa},\lambda\right\} }$.\end{prop}
\begin{proof}
By the proof of Proposition \ref{prop: many ultrafilters from IP}
and Remark \ref{rem:IP}.
\end{proof}

\subsection{Example of (\ref{enu:lambda}).}

In this section we are going to construct an example of a theory $T$
with $f_{T}\left(\kappa,\lambda\right)=\lambda$. The idea is to start
with the random graph and circularize it in order to ensure that any
non-forking type $p\in S^{\nf}\left(N,M\right)$ can be $R$-connected
to at most one point of $N$.
\begin{defn}
\label{def:phi is in circulization}Suppose $L$ is a relational language
which includes a binary relation symbol $R$. For a quantifier free
$L$-formula $\psi\left(\bar{x};\bar{y}\right)$ and atomic formulas
$\theta_{0}\left(\bar{x};\bar{y}_{0}\right)$, $\theta_{1}\left(\bar{x},\bar{y}_{1}\right)$,
where $\lg\left(\bar{x}\right)>0$, and both $\bar{x}$ and $\bar{y}_{i}$
occur in them, define the formula:
\begin{eqnarray*}
\varphi_{\psi}^{\theta_{0},\theta_{1}}\left(\bar{x};\bar{y}'\right) & =\\
\varphi_{\psi}^{\theta_{0},\theta_{1}}\left(\bar{x};\bar{y},\bar{y}_{0},\bar{y}_{1},z_{0},z_{1},z_{2}\right) & = & \theta_{0}\left(\bar{x},\bar{y}_{0}\right)\land\theta_{1}\left(\bar{x},\bar{y}_{1}\right)\land\\
 &  & \psi\left(\bar{x},\bar{y}\right)\land\\
 &  & \bigwedge_{i<j<3}R\left(z_{i},z_{j}\right)\land\bigwedge_{i<3,y\in\bar{y}\bar{y}_{0}\bar{y}_{1}}R\left(z_{i},y\right)\\
 &  & \bar{y}_{0}\neq\bar{y}_{1}.
\end{eqnarray*}

\end{defn}
So $z_{0},z_{1},z_{2}$ form a triangle and are connected to all other
parameters. The reason for this will be made clearer in the proof
of Claim \ref{cla:IsolateType}.
\begin{defn}
For a countable first-order relational language $L$ containing a
binary relation symbol $R$, Let $\F\left(L\right)$ be the set of
all formulas of the form $\varphi_{\psi}^{\theta_{0},\theta_{1}}$
from $L$ as above. Let $L_{0}=\left\{ R\right\} $ where $R$ is
a binary relation symbol. Let $T_{0}$ say that $R$ is a graph (symmetric
and non-reflexive). Let $T=\circlearrowright_{T_{0},L_{0},\F}$. \end{defn}
\begin{claim}
\label{cla:IsolateType}Let $b\in M$. Let $p_{b}\left(z\right)$
be a non-algebraic type over $M$ in one variable saying that $R\left(z,a\right)$
just when $a=b$. Then $p_{b}$ isolates a complete type over $M$. \end{claim}
\begin{proof}
We will show:
\begin{enumerate}
\item $p_{b}\upharpoonright L_{0}$ is complete.
\item If $L\supseteq L_{0}$ is some subset of $L_{\omega}$ and for all
atomic formulas $\theta\left(z\right)\in L\backslash L_{0}$ over
$M$, $p_{b}\left(z\right)\models\neg\theta\left(z\right)$, then
for all $\varphi\in L$ used in the circularization (as in Definition
\ref{def:phi is in circulization}) and atomic formulas $\theta\left(z,\bar{y}\right)\in L\left[\varphi\right]\backslash L$
and $\bar{c}\in M^{\lg\left(\bar{y}\right)}$, $p_{b}\left(z\right)\models\neg\theta\left(z,\bar{c}\right)$.
\end{enumerate}
From (1) and (2) it follows by induction that $p_{b}$ is complete.

(1) is immediate.

(2): Suppose $\theta\left(z,\bar{y}\right)$ is an atomic formula
in $L\left[\varphi\right]\backslash L$. Then it is of the form $C\left[\varphi\right]\left(\ldots\right)$
where $\varphi=\varphi_{\psi}^{\theta_{0},\theta_{1}}\left(\bar{x};\bar{y}'\right)$
for some $\psi\left(\bar{x};\bar{y}\right)$ and $\theta_{i}\left(\bar{x};\bar{y}_{i}\right)$
from $L$. Suppose $z$ appears in $\theta\left(z,\bar{y}\right)$
among the index variables. Then by the choice of $\varphi$, it follows
that $\theta\left(z,\bar{c}\right)$ implies that $z$ is $R$-connected
to at least two different elements from $M$, and this contradicts
the choice of $p_{b}$ (this is why we added the extra parameters
forming an $R$-triangle in Definition \ref{def:phi is in circulization}).
So assume that $z$ appears only in the main variables.
\begin{casenv}
\item One of $\theta_{0}$, $\theta_{1}$ is not from $L_{0}$, say $\theta_{0}$.

Since $C\left[\varphi\right]\left(\bar{y}',\bar{x}_{1},\bar{x}_{2},\bar{x}_{3}\right)\models\bigwedge\varphi\left(\bar{x}_{i},\bar{y}'\right)$,
and $p_{b}\left(z\right)\models\neg\theta_{0}\left(\ldots z\ldots\right)$
by induction (this notation means: substituting some variables of
$\theta_{0}$ with $z$, and putting parameters from $M$ elsewhere),
$p_{b}\left(z\right)\models\neg\theta\left(z,\bar{c}\right)$.
\item Both $\theta_{0},\theta_{1}\in L_{0}$. 

Suppose $\bar{c}\in M^{\lg\left(\bar{y}'\right)}$
and show that $p_{b}\left(z\right)\models\neg C\left[\varphi\right]\left(\bar{c};\ldots z\ldots\right)$.
There are two possibilities for $\theta_{i}$: $R\left(z,y\right)$
and $z=y$. If $C\left[\varphi\right]\left(\bar{c};\ldots z\ldots\right)$
holds, then we would get that either $R\left(z,c_{0}\right)\land R\left(z,c_{1}\right)$
for some $c_{0}\neq c_{1}\in M$, or some equation $x=s'$ for $s'\in M$
is in $p_{b}$ (here we use the fact that both $x$ and $\bar{y}_{i}$
occur in $\theta_{0},\theta_{1}$) --- contradiction.
\end{casenv}
\end{proof}
\begin{claim}
$f_{T}\left(\kappa,\lambda\right)\geq\lambda$. \end{claim}
\begin{proof}
Let $M\prec N\models T$, $\left|M\right|=\kappa,\left|N\right|=\lambda$.
For each $b\in M$, let $p_{b}$ be the type defined in the previous
claim. Then $p_{b}$ extends naturally to a global type $q_{b}$ (i.e.
the type over $\C$ that is $R$-connected only to $b$). This type
does not divide over $M$ (in fact it does not divide over $\emptyset$).
This is by Claim \ref{cla:NotDiv} and the proof of Claim \ref{cla:IsolateType}
(all atomic formulas in $L_{n}$ have exactly the same truth value
for $n>0$). \end{proof}
\begin{claim}
\label{cla:f(k,l)=00003Dl}$f_{T}^{n}\left(\kappa,\lambda\right)=\lambda$
for all $n$ and all $\lambda\geq2^{2^{\kappa}}$.\end{claim}
\begin{proof}
Suppose $f_{T}^{n}\left(\kappa,\lambda\right)>\lambda$. Let $M\prec N\models T$
where $\left|M\right|=\kappa,\left|N\right|=\lambda$ and $\left|S_{n}^{\nf}\left(N,M\right)\right|>\lambda$.

Let $\left\{ p_{i}\left(\bar{x}\right)\left|\,i<\lambda^{+}\right.\right\} \subseteq S_{n}^{\nf}\left(N,M\right)$
be pairwise distinct. By possibly replacing $\bar{x}$ with a sub-tuple
and throwing away some $i$'s, we may assume that for all $i<\lambda^{+}$,
$p_{i}\models\bar{x}\cap M=\emptyset$. Since $\lambda\geq2^{2^{\kappa}}$,
we may assume that for all $i<\lambda^{+}$, $p_{i}$ is not finitely
satisfiable in $M$.

Then, an easy computation shows that there must be some some $i<\lambda^{+}$
such that $p_{i}$ contains two positive occurrences of atomic formulas
$\theta_{0}\left(\bar{x},\bar{a}_{0}\right)$ and $\theta_{1}\left(\bar{x},\bar{a}_{1}\right)$
for some $\bar{a}_{0}\neq\bar{a}_{1}\in N$. Let $p=p_{i}$. There
is some quantifier free formula $\psi\left(\bar{x},\bar{c}\right)\in p$
such that $\psi$ is not realized in $M$. Let $\bar{a}$ be the tuple
of parameters $\left\langle \bar{c},\bar{a}_{0},\bar{a}_{1}\right\rangle $
and let $d_{0},d_{1},d_{2}\in N$ be an $R$-triangle such that $R\left(d_{i},a\right)$
for all $a\in\bar{a}$. Finally, let $\bar{a}'=\bar{a}d\cap M$ and
$\varphi_{\psi}^{\theta_{0},\theta_{1}}\left(\bar{x};\bar{c},\bar{a}_{0},\bar{a}_{1},d\right)\land\bar{x}\cap\bar{a}'=\emptyset\in p$
forks over $M$ by Claim \ref{cla:ForkingIffFS}.
\end{proof}

\subsection{Example of (\ref{enu:lambda ^omega}).\protect \\
}

In this subsection we prove the following Proposition:
\begin{prop}
\label{prop:lambda^aleph_0}For any theory $T$, there is a theory
$T_{*}$ such that $f_{T_{*}}\left(\kappa,\lambda\right)=f_{T}\left(\kappa,\lambda\right)^{\aleph_{0}}$
for all $\lambda\geq\kappa$.
\end{prop}
Let $T$ be a theory in the language $L$ and assume that $T$ eliminates
quantifiers. For each $n<\omega$, let $L_{n}$ be a copy of $L$
such that $L_{n}\cap L_{m}=\emptyset$ for $n<m$, and $L_{n}=\left\{ R_{n}\left|\,R\in L\right.\right\} $.
Let $\left\langle M_{n}\left|\,n<\omega\right.\right\rangle $ be
a sequence of models of $T$. We define a structure $M$ in the language
$\left\{ P_{n}\left(x\right),Q\left(x\right),f_{n}:Q\to P_{n}\left|\,n<\omega\right.\right\} \cup\bigcup L_{n}$:
\begin{enumerate}
\item $M=\bigsqcup_{n<\omega}M_{n}\sqcup\left(\prod_{n<\omega}M_{n}\right)$
($\sqcup$ means disjoint union).
\item $P_{n}^{M}=M_{n}$, $Q^{M}=\prod_{n<\omega}M_{n}$
\item If $R\left(\bar{x}\right)\in L\left(T\right)$ then for every $n<\omega$,
$R_{n}^{M}\subseteq\left(P_{n}^{M}\right){}^{\lg\left(\bar{x}\right)}$
and $P_{n}^{M}$ is the structure $M_{n}$.
\item $f_{n}^{M}:\,Q^{M}\to P_{n}^{M}$, $f_{n}^{M}\left(\eta\right)=\eta\left(n\right)$
--- the projection onto the $n$-th coordinate.
\end{enumerate}
Let $T_{*}=\mbox{Th}(M)$.
\begin{rem}
\label{rem: Models of T*} The following properties are easy to check
by back-and-forth:
\begin{enumerate}
\item Doing the same construction with respect to any sequence of models
$\left\langle M_{n}\left|\,n<\omega\right.\right\rangle $ of $T$
gives the same $T_{*}$.
\item Moreover, if we have $M_{n}\preceq N_{n}$ for all $n$ and do the
construction, then $M\preceq N$.
\item $T_{*}$ eliminates quantifiers.
\end{enumerate}
\end{rem}
Now let $M\preceq N\models T$ with $\left|M\right|=\kappa,\,\left|N\right|=\lambda$.
\begin{lem}
\label{lem: non-forking projections } Given $p\left(x\right)\in S_{1}\left(N\right)$
such that $Q\left(x\right)\in p$ , for each $n<\omega$ we let $p_{n}\left(y\right)=\left\{ \varphi\left(y\right)\left|\,\varphi\in L_{n},\,\varphi\left(f_{n}\left(x\right)\right)\in p\right.\right\} $.
\begin{enumerate}
\item $p\left(x\right)$ is equivalent to $\bigcup_{n<\omega}p_{n}\left(f_{n}\left(x\right)\right)$.
\item For each $n<\omega$, let $q_{n}\left(y\right)$ be a complete $L_{n}$-type
over $P_{n}^{N}$. 

Then the type $\left(\bigcup_{n<\omega}q_{n}\left(f_{n}\left(x\right)\right)\right)\cup\left\{ Q\left(x\right)\right\} $
is consistent and complete.
\item $P_{n}$ is stably embedded and the induced structure on $P_{n}$
is just the $L_{n}$-structure. Moreover, for any $n<\omega$ and
$L_{*}$-formula $\varphi\left(\bar{x},\bar{y}_{1},\bar{y}_{2},\bar{z}\right)$
there is some $L_{n}$-formula $\psi\left(\bar{x},\bar{y}_{1},\bar{z}'\right)$
such that for any e $\bar{c}_{1}\in P_{n}$, $\bar{c}_{2}\in\bigcup_{m\neq n}P_{m}$
and $\bar{d}\in Q$, the set $\left\{ \bar{a}\in P_{n}\left|\,\models\varphi\left(\bar{a},\bar{c}_{1},\bar{c}_{2},\bar{d}\right)\right.\right\} =\bigcup\left\{ \bar{a}\in P_{n}\left|\,\models\psi\left(\bar{a},\bar{c}_{1},f_{n}\left(\bar{d}\right)\right)\right.\right\} $.
\item $p(x)$ forks over $M$ if and only if for some $n<\omega$, $p_{n}\left(y\right)\upharpoonright L_{n}$
forks over $P_{n}^{M}$ (in the sense of $T$).
\end{enumerate}
\end{lem}
\begin{proof}
(1), (2) and (3) follows by quantifier elimination and (4) follows
from (1)--(3).
\end{proof}

\begin{proof}
(of Proposition \ref{prop:lambda^aleph_0}). We may assume that $T$
eliminates quantifiers (by taking its Morleyzation). Consider $T_{*}$as
above, and let us compute $f_{T_{*}}\left(\kappa,\lambda\right)$.
Let $M\preceq N\models T_{*}$.

Let $S_{n}=\left\{ p\in S^{\nf}\left(N,M\right)\left|\,P_{n}\left(x\right)\in p\right.\right\} $.

From Lemma \ref{lem: non-forking projections }, it follows that $\left|S_{n}\right|=\left|S^{\nf,L_{n}}\left(P_{n}^{N},P_{n}^{M}\right)\right|$.

Let $S_{Q}=\left\{ p\in S^{\nf}\left(N,M\right)\left|\,Q\left(x\right)\in p\right.\right\} $.

From Lemma \ref{lem: non-forking projections }, it follows that $\left|S_{Q}\right|=\prod_{n<\omega}\left|S^{\nf,L_{n}}\left(P_{n}^{N},P_{n}^{M}\right)\right|$.

Let $S_{\neg}=\left\{ p\in S^{\nf}\left(N,M\right)\left|\,\neg Q\left(x\right),\forall n<\omega\left(\neg P_{n}\left(x\right)\right)\right.\right\} $.

Since there is no structure on elements outside of all the $P_{n}$
and $Q$, $\left|S_{\neg}\right|\leq\left|M\right|$.

Note that $S^{\nf}\left(N,M\right)=\bigcup_{n<\omega}S_{n}\cup S_{Q}\cup S_{\neg}$.
From this and Remark \ref{rem: Models of T*}(2), it follows that
$f_{T^{*}}\left(\kappa,\lambda\right)=f_{T}\left(\kappa,\lambda\right)^{\aleph_{0}}$. \end{proof}
\begin{rem}
This analysis easily generalizes to show that $f_{T_{*}}^{n}\left(\kappa,\lambda\right)=f_{T}^{n}\left(\kappa,\lambda\right)^{\aleph_{0}}$.
\end{rem}

\subsection{Examples of (\ref{enu:ded lambda}) and (\ref{enu:ded lambda ^ omega}).\protect \\
}

Here we construct an example of a theory $T$ with $f_{T}\left(\kappa,\lambda\right)=\ded\lambda$.
The idea is that we start with an ordered random graph, and we circularize
in order to ensure that for any $p\in S^{\nf}\left(N,M\right)$ there
is some cut of $N$ such that $R\left(x,a\right)$ is in $p$ if any
only if $a$ is in the cut.
\begin{notation}

\begin{enumerate}
\item Here the language $L$ contains an order relation $<$ which induces
the natural lexicographic order on tuples, so abusing notation, we
may write $\bar{y}<\bar{z}$.
\item In this section, we say that two atomic formulas $\theta_{1}\left(\bar{x};\bar{y}_{1}\right)$
and $\theta_{2}\left(\bar{x};\bar{y}_{2}\right)$ are different when
the relation symbol in different (rather than just the variables are
different).
\item Also, when we say atomic formula in the definition below, we mean
that it does \underline{not} use the order relation $<$.
\end{enumerate}
\end{notation}
\begin{defn}
Suppose $L$ is a relational language which includes a binary relation
symbol $R$, a unary predicate $P$ and an order relation $<$.
\begin{enumerate}
\item For a quantifier free $L$-formula $\psi\left(\bar{x};\bar{y}\right)$
and two \underline{different} atomic formulas $\theta_{0}\left(\bar{x};\bar{y}_{0}\right)$,
$\theta_{1}\left(\bar{x},\bar{y}_{1}\right)$, where $\lg\left(\bar{x}\right)>0$,
and both $\bar{x}$ and $\bar{y}_{i}$ occur in them, define the formula,
define the formula
\begin{eqnarray*}
\varphi_{\psi}^{\theta_{0},\theta_{1}}\left(\bar{x};\bar{y}'\right) & =\\
\varphi_{\psi}^{\theta_{0},\theta_{1}}\left(\bar{x};\bar{y},\bar{y}_{0},\bar{y}_{1},z_{0},z_{1}\right) & = & \theta_{0}\left(\bar{x},\bar{y}_{0}\right)\land\theta_{1}\left(\bar{x},\bar{y}_{1}\right)\land\\
 &  & \psi\left(\bar{x},\bar{y}\right)\land\\
 &  & z_{0}<z_{1}\land P\left(z_{0}\right)\land P\left(z_{1}\right)\land\\
 &  & \bigwedge_{y\in\bar{y}\bar{y}_{0}\bar{y}_{1},i<2}\left(y\neq z_{i}\right)\land R\left(y,z_{1}\right)\land\neg R\left(y,z_{0}\right).
\end{eqnarray*}

\item For an $L$-formula $\psi\left(\bar{x};\bar{y}\right)$ and an atomic
formula $\theta\left(\bar{x};\bar{y}_{0}\right)$ (in which $\bar{y}_{0}$
appears) , define the formula
\begin{eqnarray*}
\varphi_{\psi}^{\theta}\left(\bar{x};\bar{y}'\right) & =\\
\varphi_{\psi}^{\theta}\left(\bar{x};\bar{y},\bar{y}_{0},\bar{y}_{1},z_{0},z_{1}\right) & = & \neg\theta\left(\bar{x},\bar{y}_{0}\right)\land\theta\left(\bar{x},\bar{y}_{1}\right)\land\\
 &  & \psi\left(\bar{x},\bar{y}\right)\land\\
 &  & z_{0}<z_{1}\land P\left(z_{0}\right)\land P\left(z_{1}\right)\land\\
 &  & \bigwedge_{y\in\bar{y}\bar{y}_{0}\bar{y}_{1},i<2}\left(y\neq z_{i}\right)\land R\left(y,z_{1}\right)\land\neg R\left(y,z_{0}\right)\\
 &  & \bar{y}_{0}<\bar{y}_{1}.
\end{eqnarray*}

\end{enumerate}
\end{defn}

\begin{defn}
For a countable first-order relational language $L$ containing a
binary relation symbol $R$, Let $\F\left(L\right)$ be the set of
all formulas from $L$ of the form $\varphi_{\psi}^{\theta_{0},\theta_{1}}$
or $\varphi_{\psi}^{\theta}$ as above. Let $L_{0}=\left\{ R,<\right\} $
where $R$ and $<$ are binary relation symbols. Let $T_{0}$ say
that $R$ is a graph and that $<$ is a linear order. Let $T=\circlearrowright_{T_{0},L_{0},\F}$.
\end{defn}
Suppose $M\models T$.
\begin{claim}
Let $I$ be initial segments in $M$. Let $p_{I}\left(x\right)$ be
a non-algebraic type over $M$ saying that $x>M$, $\neg P\left(x\right)$
and $R\left(x,a\right)$ just when $a\in I$. Then $p_{I}$ isolates
a complete type over $M$. \end{claim}
\begin{proof}
In fact, $p_{I}\upharpoonright L_{0}$ is complete, and for all atomic
formulas $\theta\left(x\right)\notin L_{0}$ over $M$, $p_{I}\models\neg\theta\left(x\right)$.
The proof is very similar to the proof of Claim \ref{cla:IsolateType}.\end{proof}
\begin{claim}
$f_{T}\left(\kappa,\lambda\right)\geq\ded\left(\lambda\right)$. \end{claim}
\begin{proof}
Let $M\prec N\models T$, $\left|M\right|=\kappa,\left|N\right|=\lambda$.
For each cut $I$ in $N$, let $p_{I}$ be the type defined in the
previous claim. Then $p_{I}$ extends naturally to a global type $q_{I}$
(i.e. the type over $\C$ defined by $p_{I'}$ where $I'=\left\{ c\in\C\left|\,\exists a\in I\left(c<a\right)\right.\right\} $).
This type does not divide over $M$ (in fact it does not divide over
$\emptyset$) by Claim \ref{cla:NotDiv} and by the proof of the previous
claim (all atomic formulas have exactly the same truth value in $L_{n}$
for $n>0$). \end{proof}
\begin{claim}
$f_{T}^{n}\left(\kappa,\lambda\right)=\ded\left(\lambda\right)$ for
all $n$ and all $\lambda\geq2^{2^{\kappa}}$.\end{claim}
\begin{proof}
Suppose $f_{T}^{n}\left(\kappa,\lambda\right)>\ded\left(\lambda\right)$.
Let $M\prec N\models T$ where $\left|M\right|=\kappa,\left|N\right|=\lambda$.

Let $\left\{ p_{i}\left(\bar{x}\right)\left|\,i<\ded\left(\lambda\right)^{+}\right.\right\} \subseteq S^{\nf}\left(N,M\right)$
is a set of pairwise distinct types. As in the proof of Claim \ref{cla:f(k,l)=00003Dl},
we may assume that $p_{i}\models\bar{x}\cap M=\emptyset$ for all
$i$, and that $p_{i}$ is not finitely satisfiable in $N$. Also
we may assume that $p_{i}\upharpoonright\left\{ <\right\} $ is constant.

Then, by the choice of $\varphi_{\psi}^{\theta_{0},\theta_{1}}$,
for every $i<\ded\left(\lambda\right)^{+}$ there is at most one atomic
formula of the form $\theta\left(\bar{x};\bar{y}\right)$ such that
there is some positive instance $\theta\left(\bar{x},\bar{a}\right)\in p_{i}$
(if not, suppose $\theta_{0}\left(\bar{x},\bar{a}_{0}\right)\land\theta_{1}\left(\bar{x},\bar{a}_{1}\right)\in p$.
There is some quantifier free formula $\psi\left(\bar{x},\bar{c}\right)\in p_{i}$
such that $\psi$ is not realized in $M$. Let $\bar{a}$ be the tuple
of parameters $\left\langle \bar{c},\bar{a}_{0},\bar{a}_{1}\right\rangle $
and let $d_{0},d_{1},d_{2}\in N$ be an $R$-triangle such that $R\left(d,b\right)$
for all $b\in\bar{a}$. Finally, let $\bar{a}'=\bar{a}d\cap M$ and
$\varphi_{\psi}^{\theta_{0},\theta_{1}}\left(\bar{x};\bar{c},\bar{a}_{0},\bar{a}_{1},d\right)\land\bar{x}\cap\bar{a}'=\emptyset\in p$
forks over $M$ by Claim \ref{cla:ForkingIffFS}).

Similarly, by the choice of $\varphi_{\psi}^{\theta}$, this formula
induces a cut $I=\left\{ \bar{a}\left|\,\theta\left(\bar{x},\bar{a}\right)\in p_{i}\right.\right\} $
.

This formula and the cut it induces determine the type. But this is
a contradiction to the definition of $\ded$.\end{proof}
\begin{cor}
There is a theory $T_{*}$ such that $f_{T_{*}}\left(\lambda,\kappa\right)=\ded\left(\lambda\right)^{\aleph_{0}}$.\end{cor}
\begin{proof}
By Proposition \ref{prop:lambda^aleph_0}.
\end{proof}

\subsection{Example of (\ref{enu:2 ^ lambda}).\protect \\
}

As a pleasant surprise to the reader who managed to get this far,
the example is just the theory of the random graph (it is $\NTPT$
and has IP, see Proposition \ref{thm:IP in NTP2}).

\subsection{\label{sub:f_T1 different than f_T2}Example of $f_{T}^{1}\left(\kappa,\lambda\right)\leq2^{2^{\kappa}}$
but $f_{T}^{2}\left(\kappa,\lambda\right)=2^{\lambda}$.\protect \\
}

Again we use circularizations, but instead of considering all formulas,
we consider only formulas with one variable.
\begin{defn}
Let $L_{0}=\left\{ =\right\} $ and $T_{0}$ be empty. Let $\F\left(L\right)$
be the set of all quantifier free partitioned formulas from $L$ of
the form $\varphi\left(x;\bar{y}\right)$ where $x$ is a singleton.
Let $T=\circlearrowright_{T_{0},L_{0},\F}$.
\end{defn}
Let $A\subseteq M\models T$. By Claim \ref{cla:ForkingIffFS} and
as in the proof of Proposition \ref{prop:f_T =00003D 2^2^kappa},
\begin{cor}
If $p\left(x\right)\in S_{1}\left(M\right)$ then $p$ does not fork
over $A$ if and only if it is finitely satisfiable in $A$. So $f_{T}^{1}\left(\kappa,\lambda\right)\leq2^{2^{\kappa}}$
for all
\end{cor}
On the other hand, if we consider types in two variables, then there
is no reason for them to fork.
\begin{claim}
$f_{T}^{2}\left(\kappa,\lambda\right)\geq2^{\lambda}$. \end{claim}
\begin{proof}
Suppose $\left|M\right|=\lambda$, so $M=\left\{ a_{i}\left|i<\lambda\right.\right\} $,
and $A\subseteq M$ of size $\kappa$. Let $q\left(z\right)\in S_{1}\left(M\right)$
be any 1-type which is finitely satisfiable in $A$ but not algebraic
over $A$. For $S\subseteq\lambda$, let $p_{S}\left(x,y\right)$
be a partial type over $M$ such that
\begin{enumerate}
\item $p_{S}\upharpoonright x=q\left(x\right)$, $p_{S}\upharpoonright y=q\left(y\right)$.
\item $R\left(x,y,a_{i}\right)\in p_{S}$ if and only if $i\in S$.
\end{enumerate}
First, $p_{S}$ is indeed a type. The proof is by induction, i.e.
one proves that $p_{S}\upharpoonright L_{0}$ is a type (which is
clear), and that if $L$ is some subset of $L_{\omega}$ such that
$p_{S}\upharpoonright L$ is a type and $\varphi\left(x;\bar{y}\right)$
is some partitioned $L$-formula with $\lg\left(x\right)=1$, then
also $p_{S}\upharpoonright L\left[\varphi\right]$ is a type, and
this follows from Claim \ref{cla:existence}.

Let $N\supseteq M$ be an $\left|A\right|^{+}$-saturated model and
$q'\supseteq q$ be a global type which is finitely satisfiable in
$A$. Fix $c\models q'|_{N}$ and $d\models q'|_{Nc}$.

We want to construct a completion $r_{S}\left(x,y\right)\in S_{2}\left(N\right)$
containing $p_{S}$ which does not divide over $A$. We start by $r_{S}\upharpoonright x=q'|_{N}\left(x\right)$,
$r_{S}\upharpoonright y=q'_{N}\left(y\right)$ and $r_{S}\upharpoonright L_{0}$
is any completion of $p_{S}\upharpoonright L_{0}$. For each atomic
formulas $\theta\left(x,y,\bar{t}\right)$ over $N$ of the form $C\left[\varphi\right]\left(\bar{t},-,-,-\right)$
(so $\bar{t}\in N$) such that $\varphi\left(x,t\right)\in q'\left(x\right)$
define $\theta\left(x,y\right)\in r_{S}$ if and only if $\theta\left(c,d\right)$
holds. This is a type (by induction again, by Claim \ref{cla:existence}
(3), but follow the proof a bit more carefully, and choose the amalgamation
of the circular orders corresponding to $\bar{t}$ according to the
choice of $c,d$). Let $r_{S}$ by any completion.

Finally, $r_{S}$ does not divide over $A$ by Claim \ref{cla:ConverseToDividing}
(by induction and by the choice of $c,d$).
\end{proof}

\section{\label{sec:on ded kappa < ded kappa ^aleph0}On $\protect\ded\kappa<\left(\protect\ded\kappa\right)^{\aleph_{0}}$}

\subsection{On $\protect\ded\left(\lambda\right)$.}
\begin{defn}
\label{def:ded}Let $\ded\left(\lambda\right)$ be the supremum of
the set
\[
\left\{ \left|I\right|\left|\,I\mbox{ is a linear order with a dense subset of size }\leq\lambda\right.\right\} .
\]
\end{defn}
\begin{fact}
\label{fac:ded inequalities-1}It is well known that $\lambda<\ded\lambda\leq\left(\ded\lambda\right)^{\aleph_{0}}\leq2^{\lambda}$.
If $\ded\lambda=2^{\lambda}$, then $\ded\lambda=\left(\ded\lambda\right)^{\aleph_{0}}=2^{\lambda}$.
This is true for $\lambda=\aleph_{0}$, or more generally for any
$\lambda$ such that $\lambda=\lambda^{<\lambda}$. So in particular
this holds for any $\lambda$ under GCH.

In addition, if $\ded\lambda$ is not attained (i.e. it is a supremum
rather than a maximum), then $\cof\left(\ded\lambda\right)>\lambda$.
See also Corollary \ref{cor:cof ded lambda < lambda}.\end{fact}
\begin{defn}
\label{def:trees}
\begin{enumerate}
\item Given a linear order $I$ and two regular cardinals $\theta,\mu$,
we say that $S$ is a\emph{ $\left(\theta,\mu\right)$-cut} when it
has cofinality $\theta$ from the left and cofinality $\mu$ from
the right.
\item By a \emph{tree} we mean a partial order $\left(T,<\right)$ such
that for every $a\in T$, $T_{<a}=\left\{ x\in T\left|\,x<a\right.\right\} $
is well ordered. By a \emph{branch} in $T$ we mean a maximally linearly
ordered subset of $T$. Its \emph{length} is its order type.
\item For two cardinals $\lambda$ and $\mu$, let $\lambda^{\left\langle \mu\right\rangle _{\tr}}$
be
\[
\sup\left\{ \kappa\left|\,\mbox{\mbox{there is some tree} }T\mbox{ with }\lambda\mbox{ many nodes and }\kappa\mbox{ branches of length }\mu\right.\right\} .
\]

\end{enumerate}
\end{defn}
\begin{rem}
\label{rem:treePower=00003Dpower}Note that $\lambda^{\left\langle \mu\right\rangle _{\tr}}\leq\lambda^{\mu}$
and if $\lambda=\lambda^{<\mu}$ then $\lambda^{\left\langle \mu\right\rangle _{\tr}}=\lambda^{\mu}$
(consider the tree $\lambda^{<\mu}$ ordered lexicographically).\end{rem}
\begin{prop}
\label{prop:definitions of ded}The following cardinalities are the
same:
\begin{enumerate}
\item $\ded\left(\lambda\right)$
\item $\sup\left\{ \kappa\left|\,\mbox{there is a linear order }I\mbox{ of size }\lambda\mbox{ with }\kappa\mbox{ many cuts}\right.\right\} $
\item $\sup\left\{ \kappa\left|\,\exists\mbox{ a regular }\mu\mbox{ and a linear order }I\mbox{ of size }\leq\lambda\mbox{ with }\kappa\mbox{ many }\left(\mu,\mu\right)\mbox{-cuts}\right.\right\} $
\item $\sup\left\{ \kappa\left|\,\exists\mbox{\mbox{ a regular }\ensuremath{\mu}\ and a tree }T\mbox{ with }\kappa\mbox{ branches of length }\mu\mbox{ and }\left|T\right|\leq\lambda\right.\right\} $
\item $\sup\left\{ \kappa\left|\,\exists\mbox{\mbox{ a limit ordinal }\ensuremath{\delta}\ and a tree }T\mbox{ with }\kappa\mbox{ branches of length }\delta\mbox{ and }\left|T\right|\leq\lambda\right.\right\} $
\item \label{enu:ded is supremum of tree exponent}$\sup\left\{ \lambda^{\left\langle \mu\right\rangle _{\tr}}\left|\,\mu\leq\lambda\mbox{ is regular}\right.\right\} $
\end{enumerate}
\end{prop}
\begin{proof}
(1)$=$(2), (4)$=$(6): obvious.

(2)$=$(3): By \cite[Theorem 3.9]{Sh818}, given a linear order $I$
and two regular cardinals $\theta\neq\mu$ the number of $\left(\theta,\mu\right)$-cuts
in $I$ is at most $\left|I\right|$. Given $I$ and a regular cardinal
$\mu$, let $D_{\mu}\left(I\right)$ be the set of $\left(\mu,\mu\right)$-cuts,
and let $D\left(I\right)$ be the set of all cuts. Suppose $\left|I\right|=\lambda$,
then $\left|D\left(I\right)\right|=\sup\left\{ \left|D_{\mu}\left(I\right)\right|\left|\,\mu=\cof\left(\mu\right)\leq\lambda\right.\right\} $
holds whenever $\left|D\left(I\right)\right|>\lambda$. By Fact \ref{fac:ded inequalities-1},
$\ded\left(\lambda\right)=\sup\{D_{\mu}\left(I\right)|\,\mu=\cof\left(\mu\right)\leq\lambda,\,\left|I\right|\leq\lambda\}$.

(2)$=$(4): Follows from \cite[Theorem 2.1(a)]{Baumgartner}.

(4)$=$(5): Obviously (5) $\geq$ (4). Suppose $T$ is a tree as in
(5). Let $\mu=\cof\left(\delta\right)$ and let $U=\left\{ \delta_{i}\left|\,i<\mu\right.\right\} $
be increasing such that $\delta=\bigcup_{i<\mu}\delta_{i}$. Let $S$
be $\left\{ a\in T\left|\,\lev\left(a\right)\in U\right.\right\} $.
Then $S$ is a subset of $T$, so a tree with the induced order. For
a branch $B\subseteq T$ of length $\delta$, let $B^{S}=B\cap S$,
then $B^{S}$ is a branch of $S$ of length $\mu$. If $B_{1}\neq B_{2}$
are branches of length $\delta$ in $T$, then let $a\in B_{1}\backslash B_{2}$,
and let $a'>a$ in $B_{1}$ be such that $\lev\left(a'\right)\in U$.
Then $a'\in B_{1}^{S}\backslash B_{2}^{S}$.
\end{proof}

\subsection{Consistency of $\protect\ded\kappa<\left(\protect\ded\kappa\right)^{\aleph_{0}}$.\protect \\
}

In \cite{KeislerSixClasses}, the following fact is mentioned (without
proof), attributed to Kunen:
\begin{rem}
{[}Kunen{]} If $\kappa^{\aleph_{0}}=\kappa$ then $\left(\ded\kappa\right)^{\aleph_{0}}=\ded\kappa$.\end{rem}
\begin{proof}
Suppose $I$ is a linear order, and $J\subseteq I$ is dense, $\left|J\right|=\kappa$.
Let $\Uu$ be a non-principal ultrafilter on $\omega$. Then the linear
order $I^{\omega}/\Uu$ has $J^{\omega}/\Uu$ as a dense subset. Now\footnote{If $A$ is infinite then $A^{\omega}/\Uu$ has size $\left|A\right|^{\aleph_{0}}$:
let $g_{n}:A^{n}\to A$ be bijections. Then take $f\in A^{\omega}$
to $\bar{f}=\left\langle g_{n}\left(f\left(0\right),\ldots,f\left(n-1\right)\right)\left|\,n<\omega\right.\right\rangle $,
so that if $f\neq g$ then $\bar{f}\neq\bar{g}$ from some point onwards,
and in particular, modulo $\Uu$.}, $\left|J^{\omega}/\Uu\right|=\kappa^{\aleph_{0}}=\kappa$ and $\left|I^{\omega}/\Uu\right|=\left|I\right|^{\aleph_{0}}$.
The remark follows from Fact \ref{fac:ded inequalities-1}.
\end{proof}
Answering a question of Keisler \cite[Problem 2]{KeislerSixClasses},
we show:
\begin{thm}
It is consistent with ZFC that $\ded\kappa<\left(\ded\kappa\right)^{\aleph_{0}}$.
\end{thm}
Our proof uses Easton forcing, so let us recall:
\begin{thm}
\label{thm:(Easton)}{[}Easton{]} Let $M$ be a transitive model of
ZFC and assume that the Generalized Continuum Hypothesis holds in
$M$. Let $F$ be a function (in $M$) whose arguments are regular
cardinals and whose values are cardinals, such that for all regular
$\kappa$ and $\lambda$:
\begin{enumerate}
\item $F\left(\kappa\right)>\kappa$
\item $F\left(\kappa\right)\leq F\left(\lambda\right)$ whenever $\kappa\leq\lambda$.
\item $\cof\left(F\left(\kappa\right)\right)>\kappa$
\end{enumerate}
Then there is a generic extension $M\left[G\right]$ of $M$ such
that $M$ and $M\left[G\right]$ have the same cardinals and cofinalities,
and for every regular $\kappa$, $M\left[G\right]\models2^{\kappa}=F\left(\kappa\right)$.
\end{thm}
See \cite[Theorem 15.18]{JechSetTheory}.

Easton forcing is a class forcing but we can just work with a set
forcing, i.e. when $F$ is a set. The following is the main claim:
\begin{claim}
\label{cla:Main}Suppose $M$ is a transitive model of ZFC that satisfies
GCH, and furthermore:
\begin{itemize}
\item $\kappa$ is a regular cardinal.
\item $\left\langle \theta_{i}\left|\,i<\kappa\right.\right\rangle $, $\left\langle \mu_{i}\left|\,i<\kappa\right.\right\rangle $
are strictly increasing sequences of cardinals and $\theta=\sup_{i<\kappa}\theta_{i}$,
$\mu=\sup_{i<\kappa}\mu_{i}$.
\item $\kappa<\theta_{0}$, $\theta_{i}<\mu_{0}$ for all $i<\kappa$.
\item $\theta_{i}$ are regular for all $i<\kappa$.
\end{itemize}
\underline{Then}, letting $P$ be Easton forcing with $F:\left\{ \theta_{i}\left|\,i<\kappa\right.\right\} \to\mathbf{card}$,
$F\left(\theta_{i}\right)=\mu_{i}$ and $G$ a generic for $P$, in
$M\left[G\right]$, $\ded\theta=\mu$ and the supremum is attained.\end{claim}
\begin{rem}
\label{rem:NoteAfterClaim}Note that in $M\left[G\right]$, we also
get by Easton's Theorem \ref{thm:(Easton)} that $2^{\theta_{i}}=\mu_{i}$;
$\cof\left(\theta\right)=\cof\left(\mu\right)=\kappa<\theta$ and
$\mu^{\kappa}>\mu$.\end{rem}
\begin{proof}
First let us show that $\ded\theta\geq\mu$. Recall,
\begin{itemize}
\item $\Add\left(\kappa,\lambda\right)$ is the forcing notion that adjoins
$\lambda$ subsets to $\kappa$, i.e. it is the set of partial functions
$p:\kappa\times\lambda\to2$ such that $\left|\dom\left(p\right)\right|<\kappa$.
\item The Easton forcing notion $P$ is the set of all elements in $\prod_{i<\kappa}\Add\left(\theta_{i},\mu_{i}\right)$
such that the for every regular cardinal $\gamma\leq\kappa$, and
for each $p\in P$, the support $s\left(p\right)$ satisfies $\left|s\left(p\right)\cap\gamma\right|<\gamma$.
\end{itemize}
If $G$ is a generic of $P$, then the projection of $G$ to $i$,
$G_{i}$, is generic in $\Add\left(\theta_{i},\mu_{i}\right)$.

For $i<\kappa$, consider the tree $T_{i}=\left(2^{<\theta_{i}}\right)^{M}$.
Since $M$ satisfies GCH, $M\left[G\right]\models\left|T_{i}\right|=\theta_{i}$.
For all $\beta<\mu_{i}$, we can define a function $\eta_{\beta}:\theta_{i}\to2$
by $\eta_{\beta}\left(\alpha\right)=p\left(\alpha,\beta\right)$ for
some $p\in G_{i}$ such that $\left(\alpha,\beta\right)\in\dom\left(p\right)$.
If $\alpha<\theta_{i}$, then $\eta_{\beta}\upharpoonright\alpha\in M$
(consider the dense set $D=\left\{ p\in\Add\left(\theta_{i},\mu_{i}\right)\left|\,\alpha\times\left\{ \beta\right\} \subseteq\dom\left(p\right)\right.\right\} $),
so for $\beta<\mu_{i}$, $\eta_{\beta}$ defines a branch of $T_{i}$,
and if $\beta_{1}\neq\beta_{2}$ then $\eta_{\beta_{1}}\neq\eta_{\beta_{2}}$.
By Proposition \ref{prop:definitions of ded} we have $\ded\theta_{i}=\mu_{i}=2^{\theta_{i}}$
in $M\left[G\right]$. Since $\ded\theta\geq\ded\theta_{i}$ for all
$i<\kappa$, we are done.

Now let us show that $\ded\left(\theta\right)\leq\mu$. Let $I$ be
some linear order such that $\left|I\right|=\theta$. For any choice
of cofinalities $\left(\kappa_{1},\kappa_{2}\right)$, we look at
the set of all $\left(\kappa_{1},\kappa_{2}\right)$-cuts of $I$,
$C_{\kappa_{1},\kappa_{2}}.$ Obviously for it to be nonempty, $\kappa_{1},\kappa_{2}\leq\theta$,
so let us assume that $\kappa_{1},\kappa_{2}\leq\theta_{i}$ for some
$i$ (note that $\theta$ is singular, so $\kappa_{1},\kappa_{2}\neq\theta$).
We map each such cut to a pair of cofinal sequences (from the left
and from the right). Hence we obtain $\left|C_{\kappa_{1},\kappa_{2}}\right|\leq\theta^{\kappa_{1}+\kappa_{2}}\leq\theta^{\theta_{i}}$.
Since $\theta\leq\mu_{0}$, $\theta^{\theta_{i}}\leq\mu_{0}^{\theta_{i}}\leq2^{\theta_{0}+\theta_{i}}=\mu_{i}<\mu$.
The number of regular cardinals below $\theta$ is $\leq\theta$,
so we are done.\end{proof}
\begin{cor}
Suppose GCH holds in $M$. Choose $\kappa=\aleph_{0}$, $\theta_{i}=\aleph_{i+1}$
and $\mu_{i}=\aleph_{\omega+i}$. Then in the generic extension, $\aleph_{\omega+\omega}=\ded\aleph_{\omega}<\left(\ded\aleph_{\omega}\right)^{\aleph_{0}}$.
In fact, since the Singular Cardinal Hypothesis holds under Easton
forcing (see \cite[Exercise 15.12]{JechSetTheory}), $\left(\ded\aleph_{\omega}\right)^{\aleph_{0}}=\aleph_{\omega+\omega+1}$.
\end{cor}

\begin{cor}
\label{cor:cof ded lambda < lambda}It is consistent with ZFC that
$\cof\left(\ded\lambda\right)<\lambda$.\end{cor}
\begin{problem}
Is it consistent with ZFC that $\ded\kappa<\left(\ded\kappa\right)^{\aleph_{0}}<2^{\kappa}$?

We remark that our construction is not sufficient for that: in the
context of Claim \ref{cla:Main}, $\left(\ded\theta\right)^{\kappa}\leq2^{\theta}$,
but $2^{\theta}=\prod_{i<\kappa}2^{\theta_{i}}\leq\prod_{i<\kappa}\mu_{i}\leq\mu^{\kappa}=\left(\ded\theta\right)^{\kappa}$.
\end{problem}

Some further properties relating the $\ded \kappa$ function and cardinal arithmetic are established in \cite{CheShe}.

\bigskip
\footnotesize
\noindent\textit{Acknowledgments.}

The first author was supported by the Marie Curie Initial Training Network in Mathematical Logic - MALOA - From MAthematical LOgic to Applications, PITN-GA-2009-238381.

The second author was partially supported by the Independent Research Start-up Grant in the Zukunftskolleg (University of Konstanz), the SFB grant 878 (University of Muenster) and the Israel Science Foundation in (grant No. 1533/14).

The third author has received funding from the
European Research Council, ERC Grant Agreement n. 338821. He would like to thank the Israel Science Foundation
for partial support of this research (Grants 710/07 and 1053/11). 
No. 1007
on the third author's list of publications. 
\bibliographystyle{alpha}
\bibliography{common}

\end{document}